\documentclass[preprint]{elsarticle}

\usepackage{amsmath}
\usepackage{amsthm}
\usepackage{amssymb}
\usepackage[utf8]{inputenc}
\usepackage[margin,marginclue,draft]{fixme}
\usepackage[usenames,dvipsnames,svgnames,table]{xcolor}
\fxsetup{margin}
\usepackage{hyperref}
\usepackage{tikz-cd}
\usepackage{enumitem}
\usepackage{multicol}
\usepackage{mathrsfs}
\usepackage{amsfonts} \usepackage{nicefrac}
\usepackage{amscd} \usepackage{a4wide}

\usepackage{url}

\begin{document}
\newtheorem{theoremi}{Theorem}
\renewcommand{\thetheoremi}{\Alph{theoremi}}

\newtheorem{theorem}{Theorem}[section]
\newtheorem{lemma}[theorem]{Lemma}
\newtheorem{fact}[theorem]{Fact}
\newtheorem{facts}[theorem]{Facts}
\newtheorem{conjecture}{Conjecture}
\newtheorem{proposition}[theorem]{Proposition}
\newtheorem{corr}[theorem]{Corollary}
\theoremstyle{definition}
\newtheorem{definition}[theorem]{Definition}
\theoremstyle{remark}
\newtheorem{notation}[theorem]{Notation}
\newtheorem{example}[theorem]{Example}
\newtheorem{claim}[theorem]{Claim}
\newtheorem{remark}[theorem]{Remark}
\newtheorem{question}{Question}
\newcommand{\rsq}[3]{{#1#2{\in}\,#3}}

\newcommand{\R}{\mathbf{R}}
\newcommand{\IN}{\mathbf{N}}
\newcommand{\Ca}{\mathcal{C}}
\newcommand{\N}{\mathbf{N}}
\newcommand{\Ps}{\mathbf{P}}
\newcommand{\IZ}{\mathbf{Z}}
\newcommand{\IQ}{\mathbf{Q}}
\newcommand{\Z}{\mathbf{Z}}
\newcommand{\IR}{\mathbf{R}}
\newcommand{\IC}{\mathcal{C}}
\newcommand{\T}{\mathrm{T}}
\newcommand{\Th}{\mathrm{Th}}
\newcommand{\Fib}{\text{Fib}}
\newcommand{\ec}{\subseteq_{\textrm{ec}}}
\newcommand{\id}{\mathrm{id}}
\newcommand{\acl}[2][]{\textrm{acl}^{#1}\left(#2\right)}
\newcommand{\op}{\mathsf{f}}
\renewcommand{\phi}{\varphi}
\newcommand{\tp}[2][]{\textrm{tp}^{#1}\left(#2\right)}
\renewcommand{\L}{\mathcal{L}}
\newcommand{\str}[1]{\mathscr{#1}}
\newcommand{\A}{\mathcal{A}}
\newcommand{\B}{\mathcal{B}}
\newcommand{\Orb}{\mathrm{Orb}}
\renewcommand{\Im}[1]{\mathrm{Im}(#1)}
\newcommand{\Imp}{\mathrm{Im}^{+}}
\newcommand{\ind}{\mathrm{ind}}
\newcommand{\tpsq}[1]{\textrm{tp}^\textrm{qf}\left(#1\right)}
\newcommand{\im}[1]{\mathrm{Im}(#1)}
\newcommand{\triv}{\mathrm{Triv}}
\newcommand{\supp}{\mathrm{supp}}
\newcommand{\I}{\mathcal{I}}

\begin{frontmatter}
  \author[add]{Quentin Lambotte\corref{cor1}%
  }
  \ead{Quentin.Lambotte@umons.ac.be}

  \author[add]{Fran\c{c}oise Point\fnref{fr}}
  \ead{point@math.univ-paris-diderot.fr}

  \cortext[cor1]{Corresponding author}

  \address[add]{D\'epartement de Math\'ematique (bâtiment De Vinci)\\
    Universit\'e de Mons\\
    20 place du Parc\\ B-7000 Mons\\ Belgium}

  \fntext[fr]{Research Director at the FRS-FNRS,
    this material is based upon work supported by the National Science
    Foundation
    under Grant No. DMS-1440140 while the author was in residence at the
    Mathematical Sciences Research Institute in Berkeley, California, during the
    Spring 2017.}

\title{On expansions of $(\IZ,+,0)$}
\date{\today}

\begin{abstract}
  Call a (strictly increasing) sequence $(r_{n})$ of natural numbers
  \emph{regular} if it satisfies the following condition: $r_{n+1}/r_{n}\to
  \theta\in\IR^{>1}\cup\{\infty\}$ and, if $\theta$ is algebraic, then $(r_{n})$
  satisfies a linear recurrence relation whose characteristic polynomial is the
  minimal
  polynomial of $\theta$. Our main result states that $(\IZ,+,0,R)$
  is superstable whenever $R$ is enumerated by a regular sequence. We give
  two proofs of this result. One relies on a result of E. Casanovas and
  M. Ziegler and the other on a quantifier elimination result. We also show that
  $(\IZ,+,0,<,R)$ is NIP whenever $R$ is enumerated by a regular sequence that is ultimately
  periodic modulo $m$ for all $m>1$.
\end{abstract}
\end{frontmatter}
\section*{Introduction}
Recently, stability properties of expansions $\str{Z}_{R}=(\IZ,+,0,R)$ of $(\IZ,+,0)$ by a
unary predicate $R$ for a subset of the integers have attracted the attention of many
researchers. Motivated by a question of A. Pillay on the induced structure on
non-trivial centralizers in the free group on two generators, D. Palac\'{i}n and
R. Sklinos proved in \cite{Pal-Sk} that for any natural number $q>1$, the
structure $\str{Z}_{\Pi_q}$ is superstable of Lascar rank $\omega$, where
$\Pi_q=\{q^n\mid n\in\IN\}$ (this was also proved independently and using different
methods by B. Poizat in \cite[Th\'eor\`eme 25]{Poi}). They also showed the same
result for $R=\{n!\mid n\in\IN\}$ and more generally for sets enumerated by sequences $(r_n)$ such that $r_{n+1}/r_n\to\infty$
and that are \textit{congruence periodic}, namely ultimately
periodic modulo $m$ for all $m>1$.
They used results of E. Casanovas and M. Ziegler
\cite{Cas-Zieg} on stable expansions by a unary predicate.  In another
direction, when $R$ is the set $\Ps$ of prime numbers, I. Kaplan and S. Shelah
showed in \cite{KapSh}, assuming Dickson's Conjecture (\cite[Conjecture
1.1]{KapSh}), that
$\str{Z}_{\Ps\cup -\Ps}$ is unstable and supersimple of Lascar rank $1$.\\
\indent In this paper, we investigate such expansions $\str{Z}_{R}$ with
$R$ interpreting a subset of the natural numbers, generalizing the
above results of D. Palac\'{i}n and R. Sklinos. Call a sequence $(r_{n}) $
\textit{regular} if it satisfies the following condition:
$r_{n+1}/r_{n}\to\theta\in\IR^{>1}\cup\{\infty\}$ and, if $\theta$ is algebraic,
$(r_{n})$ follows a linear recurrence relation whose characteristic polynomial is the
minimal polynomial of $\theta$.  For such regular sequences, we show the superstability of the corresponding expansion.
\begin{theoremi}[{Theorem \ref{theorem-sup}}]\label{theorem-a}
  Assume that $R$ is enumerated by a regular sequence. Then $\Th(\str{Z}_{R})$ is superstable
  of Lascar rank $\omega$.
\end{theoremi}
We give two proofs of this theorem. The first one follows the strategy of D. Palac\'{i}n and
R. Sklinos. The second one is based on a quantifier elimination result for such expansions.\\
\indent Expansions of Presburger arithmetic by regular sequences have been
already studied and shown to be tame. At the end of the paper, we give a proof
of an announced result in \cite{aschenbrenner,aschenbrenner2013} that expansions
of Presburger arithmetic by $\Pi_{q}$, $\{n!\mid n\in\IN\}$ or the set of Fibonacci numbers are NIP.
\begin{theoremi}[{Corollary \ref{corollary-nip}}]
  Assume that $R$ is interpreted by a congruence periodic regular sequence. Then
  $\Th(\IZ,+,0,<,R)$ is NIP.
\end{theoremi}

Let us now outline the content of the paper. From now on, we assume that $R$ is interpreted by a regular sequence.\\
\indent In Section \ref{section-regseq}, we give the first proof of Theorem \ref{theorem-a} by applying the result of E. Casanovas and M. Ziegler. In our context, their result \cite[Proposition 3.1]{Cas-Zieg} reads as follows. An expansion of the
form $\str{Z}_{R}$ is superstable whenever $R$ is bounded and the induced
structure on $R$ by
$\str{Z}_{R}$ is superstable (see Definition \ref{definition-ind} and Theorem
\ref{thm-casanovas}, and the comments after these). In fact, we only need to check that the induced
structure on $R$ by equations in $\str{Z}_{R}$ is superstable. So we analyze sets of the form $X_{\bar{a}}=\{(r_{n_{1}},\ldots,r_{n_{k}})\in
R^{k}\mid a_{1}r_{n_{1}}+\cdots+a_{k}r_{n_{k}}=0\}$, where $k\ge 1$ and $\bar{a}\in\IZ$.\\
\indent We show, in Proposition \ref{proposition-sol-eq}, that for all $\bar{a}\in\IZ^{k}$ there is $c_{\bar{a}}\in\IN$ such that if $(r_{n_{1}},\ldots,r_{n_{k}})\in X_{\bar{a}}$, then
\[\max\{|n_{i}-n_{j}|\mid1\le i,j\le k\}\le c_{\bar{a}},\]
unless there exists $I\subsetneq\{1,\ldots,k\}$ such that $\sum_{i\in I}a_{i}r_{n_{i}}=0$.
The proof of this proposition relies on the following property: sets of the form
$\{r_{n}\in R\mid a'_{0}r_{n}+a'_{1}r_{n+1}+\cdots +a'_{\ell}r_{n+\ell}=0\}$ are either finite or $R$, where $\bar{a}'\in\IZ^{\ell+1}$ and $\ell\in\IN$. The analysis of the sets $X_{\bar{a}}$ allows us to show that the induced structure on $R$ is definably interpreted in the superstable structure $(\IN,S,S^{-1},0)$, where $S(n)=n+1$, $S^{-1}(n+1)=n$ and $S^{-1}(0)=0$. Thus, the induced structure is superstable.\\
\indent Recall that a subset of $\IN$ is \emph{piecewise syndetic} if it
contains arbitrarily long sequences with bounded gaps. We use again Proposition \ref{proposition-sol-eq} to show that we cannot bound the length of expansions in base $R$ of natural numbers. In other words, we show that any set
of the form
\[\{z\in\IZ\mid z=a_{1}r_{n_{1}}+\cdots+a_{k}r_{n_{k}}\text{ for some }
  (r_{n_{1}},\ldots,r_{n_{k}})\in R^{k}\}\cap\IN\]
is not piecewise syndetic. This allows us to prove that $R$ is bounded, see Subsection \ref{section-bounded} and Theorem \ref{theorem-cover-z} therein.\\
\indent In Section \ref{section-axiomatization}, we give the second proof of Theorem \ref{theorem-sup}, using a quantifier elimination result in a language extending $\L_{R}=\{0,+,R\}$.\\
\indent We first axiomatize the theory of $\str{Z}_{R}$ in an enriched language
$\L$: to $\L_{R}$ we add new predicates interpreted in $\str{Z}_{R}$ by certain
existentially defined sets. For instance,
\[\{z\in\IZ\mid z=a_{1}r_{n_{1}}+\cdots+a_{k}r_{n_{k}}\text{ for some }
  (r_{n_{1}},\ldots,r_{n_{k}})\in R^{k}\}\]
will be quantifier-free $\L$-definable. We let $T_{R}$ be this $\L$-axiomatization of $\str{Z}_{R}$.
\begin{theoremi}[Theorem \ref{thm-eq} and Corollary \ref{cor-eq}]
  The $\L$-theory $T_{R}$ has quantifier elimination and is complete.
\end{theoremi}
\indent This quantifier elimination result allows us to prove that $T_{R}$ is
superstable by counting types. As a consequence, we recover the superstability of $\Th(\str{Z}_{R})$.\\
\indent We conclude this paper by Subsections \ref{section-decid} and
\ref{section-nip}, where we respectively show that, when $R$ is interpreted by a
congruence periodic regular sequence, $T_{R}$ is decidable (when $\theta$
can be computed effectively and the congruence periodicity is effective) and
$\Th(\IZ,+,0,<,R)$ is NIP. This last result relies on a quantifier elimination
result of the second author in \cite{P} for expansions of Presburger arithmetic
by so-called sparse predicates, introduced by A. L. S\"emenov.\\
\newline
\indent Independently of our work \cite{PointO}, G. Conant published a
paper \cite{conant} on sparsity notions and stability for sets of
integers. There he defines the notion of a \emph{geometrically sparse}
set $R$ (see \cite[Definition 6.2]{conant}). For such a set $R$, he
proves superstability of $(\Z,+,0,R)$ and calculates its Lascar rank
\cite[Theorem 7.1]{conant}. So there is an overlap between his result and our
Theorem \ref{theorem-sup} (see also \cite{Lam}); we give an account of this
overlap at the end of Section \ref{section-regseq}. We also point out that, in
the first version of this paper, our main result had an extra hypothesis on
regular sequences, namely that they were congruence periodic. This hypothesis was necessary to understand the trace, on $R$, of
congruence relations. However, G. Conant showed that in some cases, the analysis
of the trace of congruence relations is not necessary and we decided to
incorporate this in Theorem \ref{theorem-sup}. This is explained after the
statement of Theorem \ref{thm-casanovas}.

\subsection*{Notation and convention}
In this section, we fix some notations and conventions for the rest of this
paper. The set of natural numbers, of integers and of real
numbers will be denoted respectively $\IN$, $\IZ$ and $\IR$. When $X$ is one of
the above sets and $a\in X$, the notations $X^{>a}$, $X^{\ge a}$ and
$X_{\infty}^{>a}$ refer respectively to the sets $\{x\in X\mid x>a\}$,
$\{x\in X\mid x\ge a\}$ and $X^{>a}\cup\{\infty\}$. For a natural number $n$, the
set $\{1,\ldots,n\}$ will be denoted $[n]$.  The cardinality of a set $A$ will
be denoted by $|A|$. Likewise, the length of a tuple $\bar{x}$ will be denoted
$|\bar{x}|$.\\
\indent Capital letters $I$, $J$ and $K$ will refer to (usually non-empty) sets
of indices. Capital letters will refer to sets and small letters will
refer to elements of a given set. For a tuple $\bar{a}$ of length $n$ and
$I\subset [n]$, $\bar{a}_{I}$ refers to the tuple $(a_{i}\mid i\in I)$. For
$n\in\IN^{>0}$, we let $\mathfrak{P}([n])$ be the set of (ordered) partitions
$\bar{I}=(I_{1},\ldots,I_{\ell})$ of $[n]$.\\
\indent A first order language will be denoted by the letter $\L$, possibly with
a subscript. An $\L$-structure will be referred to by a round letter and its
domain by the corresponding capital letter. For instance $\str{M}$ is an
$\L$-structure whose domain is $M$. For an element $a$ of $M$ and $A\subset M$,
the notations $\acl[\L]{a/A}$, $\tp[\L]{a/A}$ mean respectively the
algebraic closure and the type of $a$ over $A$ in
$\str{M}$. If $R\in\L$ is a $n$-ary predicate symbol, the set
$\{\bar{a}\in M^{n}\mid \str{M}\models R(\bar a)\}$
will be denoted $R(M^{n})$ or simply $R$ when there is no confusion.\\
\indent We make the following (usual) abuse of notations. When $R$ is a unary
predicate symbol, expressions of the form $\exists x\in R\,\phi(x)$ and
$\forall x\in R\,\phi(x)$ respectively mean $\exists x\,(R(x)\wedge \phi(x))$
and $\forall x\,(R(x)\rightarrow\phi(x))$. An expression of the form $x>c$,
where $c\in\IN$, is an abbreviation for
$\bigwedge_{i=0}^{c}x\neq i$.\\
\indent For each $n\in\IN^{>1}$, let $D_{n}$ be a unary predicate. We let
$\L_{g}=\{+,-,0,D_{n}\mid n>1\}$ and $\L_{S}=\{S,S^{-1},c\}$, where $S$ and $S^{-1}$
are unary function symbols and $c$ is a
constant symbol. An abelian group $(G,+,-,0)$ will always be expanded to an
$\L_{g}$-structure as follows: for each
$n\in\IN^{>1}$, the symbol $D_{n}$ is interpreted as the set
$\{x\in G\mid (G,+,-,0)\models\exists y\, x=ny\}$.

\section{Expansion of $(\IZ,+,-,0)$ by a regular sequence}\label{section-regseq}
In this section, we consider expansions of $(\IZ,+,-,0)$ by a unary predicate
$R$ interpreting an infinite subset of $\IN$. We let $(r_{n})$ be the unique
(strictly increasing) enumeration of $R(\IZ)$. When there is no risk of
confusion, we will also denote $(r_{n})$ by $R$. The main result of this section
is the superstability of the expansion $\str{Z}_{R}=(\IZ,+,-,0,R)$ when
$(r_{n})$ is a \emph{regular sequence}, defined below.
\begin{definition}\label{definition-char-pol}
  Let $R=(r_{n})$ be a sequence of natural numbers that satisfy a \emph{linear
    recurrence relation}: there are $a_{0},\ldots,a_{k-1}\in\IQ$, with
  $k\in\IN^{\ge 1}$ minimal, such that for all $n\in\IN$,
  \[r_{n+k}=\sum_{i=0}^{k-1}a_{i}r_{n+i}.\]
  The polynomial $P_{R}$ defined by $P_{R}(X)=X^{k}-\sum_{i=0}^{k-1}a_{i}X^{i}$
  is called the \emph{characteristic polynomial} and the numbers
  $r_{0},\ldots,r_{k-1}$ the \emph{initial conditions}.
\end{definition}
\begin{definition}\label{definition-ref-seq}
  Let $R=(r_{n})$ be a sequence of natural numbers. We say that $R$ is
  \emph{regular} if it satisfies the following property:
  $r_{n+1}/r_{n}\to\theta\in\IR^{>1}_{\infty}$ and, if $\theta$ is algebraic
  (over $\IQ$), then $R$ satisfies a linear recurrence relation whose
  characteristic polynomial $P_{R}$ is the minimal polynomial of $\theta$.
\end{definition}
\begin{remark}
  Let $R=(r_{n})$ be defined by a recurrence relation whose characteristic
  polynomial is $P_{R}$. A. Fiorenza and G. Vincenzi, in \cite{fior1,fior2}
  provide a necessary and sufficient condition on $P_{R}$ and the initial
  conditions of $R$ for the existence of $\lim\limits_{n\to\infty}r_{n+1}/r_{n}$
  (see \cite[Theorem 2.3]{fior1}).
\end{remark}
We define an action of $\IZ[X]$ on the set of sequences of integers. We let $X$
act as the shift $\sigma$: for all $n\in\IN$, $\sigma(s_{n})=s_{n+1}$. Likewise,
$X^{i}$ acts as $\sigma^{i}$. We extend this by linearity: if
$Q(X)=\sum^{d}_{i=0}a_{i}X^{i}$, then $Q$ acts as
$\sum_{i=0}^{d}a_{i}\sigma^{i}$. This action has the following property: for all
$Q,Q'\in\IZ[X]$, $QQ'$ acts as the action of $Q'$ followed by the
action of $Q$. Let $Q\cdot$ denote the action of $Q$. Note that
$P_{R}\cdot(r_{n})=(0)$.\\
\indent We will need the following well-known result on linear recurrence
relations.
\begin{proposition}\label{proposition-rec-rel-div}
  Let $R=(r_{n})$ be a linear recurrence relation and let $Q\in\IQ[X]$,
  $Q(X)=\sum_{i=0}^{d}a_{i}X^{i}$. The following are equivalent
  \begin{enumerate}
  \item $P_{R}$ divides $Q$ (in $\IQ[X]$);
  \item $Q\cdot (r_{n})=(0)$, that is for all $n\in\IN$,
    $a_{0}r_{n}+a_{1}r_{n+1}+\cdots+a_{d}r_{n+d}=0$.
  \end{enumerate}
\end{proposition}
\begin{proof}
  By the euclidean algorithm (in $\IQ[X]$, see \cite[Theorem 2.14]{jacob}),
  we have $Q=Q_{1}P_{R}+Q_{2}$, for
  some $Q_{1},Q_{2}\in\IQ[X]$ with $\deg(Q_{2})<\deg(P_{R})$.\\
  \indent First assume that $P_{R}$ divides $Q$. Then $Q=Q_{1}P_{R}$. Thus,
  $Q\cdot(r_{n})=(Q_{1}P_{R})\cdot(r_{n})=Q_{1}
  \cdot(P_{R}\cdot(r_{n}))=Q_{1}\cdot(0)=(0)$. This
  implies that $Q\cdot(r_{n})=(0)$.\\
  \indent Second assume that $Q\cdot(r_{n})=(0)$. Since $P_{R}\cdot(r_{n})=(0)$,
  we get that $Q_{2}\cdot(r_{n})=(0)$, which contradicts the minimality of
  $P_{R}$, unless $Q_{2}=0$. So $P_{R}$ must divide $Q$.
\end{proof}

Here is a list of examples of regular sequences (we will come back to some of
these examples at the end of this section):

\begin{itemize}
\item $(n!)$;
\item $(q^{n})$, where $q\in\IN^{>1}$;
\item the sequence $(\lfloor{\pi^{n}}\rfloor)$, where $\lfloor x\rfloor$ denotes the
  integer part of $x$;
\item the Fibonacci sequence as defined by $r_{0}=1$, $r_{1}=2$ and
  $r_{n+2}=r_{n+1}+r_{n}$ for all $n\in \IN$;
\item the sequence given by the linear recurrence relation
  $r_{n+2}=5r_{n+1}+7r_{n}$, $r_{1}=1$ and $r_{0}=0$.
\end{itemize}
\indent\indent Now, let us state the main result of this section.
\begin{theorem}\label{theorem-sup}
  Assume $R(\IZ)$ is enumerated by a regular sequence. Then $\Th(\str{Z}_{R})$ is superstable of Lascar rank $\omega$.
\end{theorem}
The proof of this theorem follows the same strategy as D. Palac\'{i}n and R. Sklinos
in \cite{Pal-Sk} which is based on the following result of E. Casanovas and
M. Ziegler \cite{Cas-Zieg}.
\begin{definition}\label{definition-ind} Let $\str{M}$ be an $\L$-structure and
  $A\subset M$, $A\neq\emptyset$.
  \begin{enumerate}
  \item We say that $\str{M}$ \emph{does not have the finite cover property} (in
    short: $\str{M}$ is nfcp) if for all formulas $\phi(x,\bar{y})$, there exists
    $k\in\IN$ such that for all $\bar{m}_{i}\in M^{|\bar{y}|}$, $i\in I$, if the set
    $X=\{\phi(x,\bar{m}_{i})\mid i\in I\}$ is $k$-consistent, then $X$ is
    consistent. Similarly, we say that $\str{M}$ \emph{has} nfcp \emph{over $A$}
    if for all formulas $\phi(x,\bar{y},\bar{z})$, there exists
    $k\in\IN$ such that for all $\bar{a}_{i}\in A^{|\bar{y}|}$ $\bar{m}_{i}\in
    M^{|\bar{z}|}$, $i\in I$, if the set
    $X=\{\phi(x,\bar{a}_{i},\bar{m}_{i})\mid i\in I\}$ is $k$-consistent, then $X$ is
    consistent.
  \item To each $\L$-formula $\phi(x_{1},\ldots,x_{n})$, we associate a new
    $n$-ary predicate $R_{\phi,n}$ and we denote by $\L_{\ind}$ the
    language
    $$\{R_{\phi,n}\mid \phi(x_{1},\ldots,x_{n})\textrm{ is an
    }\L\textrm{-formula}\}.$$ The \emph{induced structure} on $A$ (by
    $\str{M}$), denoted $A_{\ind}$, is the $\L_{\ind}$-structure whose domain is
    $A$ and $R_{\phi,n}(A)=\phi(M^{n})\cap A^{n}$. Similarly, we define
    $A_{\ind}^{0}$ to be the induced structure on $A$ by equations ($\L_{\ind}$
    is replaced by $\L_{\ind}^{0}$, which contains the symbols $R_{\phi,n}$,
    where $\phi(x_{1},\ldots,x_{n})$ is a boolean combination of
    equations\footnote{By an equation $\phi(x_{1},\ldots,x_{n})$, we mean an
      atomic formula of the form $t(x_{1},\ldots,x_{n})=t'(x_{1},\ldots,x_{n})$ where
      $t$ and $t'$ are $\L$-terms.} in $\L$).
  \item Let $R$ be a unary predicate not in $\L$ and let $\L_{R}$ be the
    language $\L\cup\{R\}$. Let $\str{M}_{A}$ denote the $\L_{R}$-expansion of
    $\str{M}$, with $R(M)=A$.
    \begin{enumerate}
    \item We say that $A$ is \emph{small} if there is an $\L_{R}$-structure
      $\str{N}_{R(N)}$ elementary equivalent to $\str{M}_{A}$ such that: for all
      finite subsets $B$ of $N$, any type in $\L$ over $B\cup R(N)$ is realized
      in $\str{N}$.
    \item We say that $A$ is \emph{bounded} if for all $\L_{R}$-formulas
      $\phi(\bar{x})$ there is an $\L$-formula $\psi(\bar{x},\bar{y})$ such that
      $\phi$ is equivalent (in $\str{M}$) to the formula
      $\rsq{Q_{1}}{y_{1}}{R}\,\ldots
      \;\rsq{Q_{n}}{y_{n}}{R}\,\psi(\bar{x},\bar{y})$, where
      $Q_{i}\in\{\exists,\forall\}$.
    \end{enumerate}
  \end{enumerate}
\end{definition}
The main result of \cite{Cas-Zieg} states that the stability of $\str{M}_{A}$ is
equivalent to that of $\str{M}$ and $A_{\ind}$, under the assumptions of nfcp
for $\str{M}$ and smallness for $A$, see \cite[Theorem A]{Cas-Zieg} for a
precise statement. To prove \cite[Theorem A]{Cas-Zieg}, E. Casanovas and
M. Ziegler first show that when $\str{M}$ is nfcp, if $A$ is small then $A$ is
bounded. Next they show an analogue of \cite[Theorem A]{Cas-Zieg} under the sole
assumption of boundness of $A$. This is actually the result used by
D. Palac\'{i}n and R. Sklinos in \cite{Pal-Sk}.
\begin{theorem}[{\cite[Proposition 3.1]{Cas-Zieg}}]\label{thm-casanovas}
  Let $\str{M}$ be an $\L$-structure and let $A\subset M$. Suppose that $A$ is
  bounded. Then for all $\lambda\ge|\L|$, if $\str{M}$ and $A_{\ind}$ are
  $\lambda$-stable, then $\str{M}_{A}$ is $\lambda$-stable.
\end{theorem}
Let $R(\IZ)\subset \IN$, enumerated by a regular sequence $(r_{n})$. Recall that
$\str{Z}$ has quantifier elimination in $\L_{g}$ (see
\cite[Theorem15.2.1]{rothmaler}) and is superstable (see
\cite[Theorem 15.4.4]{rothmaler}). So, in order to show that $\str{Z}_{R}$ is
superstable, we only need to show that $R$ is small and that $R_{\ind}$ is
superstable. For the latter, the study of $R_{\ind}$ is reduced to the study of
the trace on $R(\IZ)$ of equations of the form $a_{1}x_{1}+\cdots+a_{n}x_{n}=0$
or divisibility relations of the form
$D_{m}(a_{1}x_{1}+\cdots+a_{n}x_{n})$. Actually, we can further reduce the
analysis of $R_{\ind}$ to the analysis of $R^{0}_{\ind}$, using the following
observation of G. Conant (see \cite[Section 5]{conant}).  For $\str{N}$ an
$\L$-structure, let $\str{N}^{1}$ be the expansion of $\str{N}$ by predicates
for \emph{all} subsets of $N$. G. Conant observed that $R_{\ind}$ is an
expansion of $R_{\ind}^{0}$ by unary predicates \cite[Corollary 5.7]{conant}. As
a consequence of this observation, if $R_{\ind}^{0}$ is definably interpreted in
a structure $\str{N}$ whose expansion $\str{N}^{1}$ is superstable, then
$R_{\ind}$ is superstable. We apply this by showing that $R_{\ind}^{0}$ is
definably interpreted in the structure $\str{N}=(\IN,S)$, where $S(n)=n+1$,
whose expansion $\str{N}^{1}$ has been shown to be superstable \cite[Proposition
5.9]{conant}. This is done as follows. We interpret $R(\IZ)$ by
$\IN$. Given an equation $a_{1}x_{1}+\cdots+a_{k}x_{k}=0$, where
$\bar{a}\in\IZ^{k}$, we interpret the set of solutions
$(r_{n_{1}},\ldots,r_{n_{k}})$ in $R(\IZ)^{k}$ as the set
$\{(n_{1},\ldots,n_{k})\in\IN^{k}\mid a_{1}r_{n_{1}}+\cdots+a_{k}r_{n_{k}}=0\}$ and we
show that it is $\L_{S}$-definable. We actually prove that sets of the form
\[
  N_{\bar{a}}=\left\{(n_{1},\ldots,n_{k})\in\IN^{k}\,\left|\, a_{1}r_{n_{1}}+\cdots+a_{k}r_{n_{k}}=0\text{
    and for all }I\subset[k],\sum_{i\in I}\right.a_{i}n_{i}\neq 0\right\}.
\]
are $\L_{S}$-definable and explain why this is enough to conclude. In order to
show that the sets $N_{\bar{a}}$ are definable, we first show that functions of
the form $\op:\IN\to \IZ:n\mapsto
b_{0}r_{n}+b_{1}\sigma(r_{n})+\cdots+b_{d}\sigma^{d}(r_{n})$, where
$\bar{b}\in\IZ$, called \emph{operators} on $R$, behave predictably: either
$\op(n)=0$ for all $n\in\IN$ or $\op(n)\neq 0$ for all but
finitely many $n\in\IN$, see Proposition \ref{proposition-operators}. So for any
operator $\op$, the set $\{n\in\IN\mid \op(n)=0\}$ is $\L_{S}$-definable. Then, we
show that the set $N_{\bar{a}}$ is controlled by
finitely many operators, see Proposition \ref{proposition-sol-eq} and Remark
\ref{remark-control}. This control has the effect to reduce the definability of
$N_{\bar{a}}$ to the definability of sets of the form $\{n\in\IN\mid \op(n)=0\}$.\\
\indent The boundness of $R$ will be a consequence of Propositions \ref{proposition-operators} and
\ref{proposition-sol-eq}: we deduce from them that a set
of the form $a+d\IN$ cannot be covered by finitely many sets of the form
$\{z+\op_{1}(n_{1})+\cdots+\op_{k}(n_{k})\mid \bar{n}\in\IN\}$, where $z\in\IZ$ and
$\op_{1},\ldots,\op_{k}$ are operators on $R(\IZ)$. This is done in Section
\ref{section-bounded}.
\subsection{The induced structure on a regular sequence}\label{section-induced}
Throughout this section we fix $R(\IZ)\subset\IN$ and we assume that $(r_{n})$
is a regular sequence and
$\theta=\lim\limits_{n\to\infty}r_{n+1}/r_{n}\in\IR^{>1}_{\infty}$.
\begin{definition}\label{definition-op}
  Let $Q\in\IZ[X]$, $Q(X)=\sum^{d}_{i=0}a_{i}X^{i}$, $\bar{a}\in\IZ^{d+1}$. The
  \emph{operator} associated to $Q$, denoted $\op_{Q}$ or simply $\op$, is the
  function
  $\op:\IN\to \IZ:n\mapsto
  a_{0}r_{n}+a_{1}\sigma(r_{n})\cdots+a_{d}\sigma^{d}(r_{n})$.
\end{definition}
The aim of this section is to understand the structure of subsets of $\IN^{k}$ of the form
\begin{align}
  \label{eq:2}
  N_{\bar{\op},z}=\{\bar{n}\in\IN^{k}\mid \op_{1}(n_{1})+\cdots+\op_{s}(n_{k})=z\},
\end{align}
where $\op_{i}$ is an operator, $i\in[s]$ and $z\in\IZ$.
\subsubsection{Operators on a regular sequence}
We start with the case of a single operator, that is with sets of the form
$N_{\op,z}$.
\begin{proposition}\label{proposition-operators}
  Let $\op$ be an operator. Then $N_{\op,0}$ is either finite or $\IN$.
\end{proposition}
The proof of this proposition follows from the following lemmas.
\begin{lemma}
  Suppose that $\theta=\infty$. Let $Q\in\IZ[X]\setminus\{0\}$. Then $N_{\op_{Q},0}$ is finite.
\end{lemma}
\begin{proof}
  Assume that $Q(X)=\sum_{i=0}^{d}a_{i}X^{i}$ and let $\op=\op_{Q}$. Then
  $\op(n)=0$ if and only if
  $a_0r_n/r_{n+d}+\cdots+a_{d-1}r_{n+d-1}/r_{n+d}+a_{d}=0$. Thus, as
  $r_{n+i}/r_{n+d}\to 0$ for all $0\leq i<d$, for all $n$ sufficiently large,
  $\op(n)\neq 0$.
\end{proof}
\begin{lemma}\label{lemma-op-fin}
  Suppose that $\theta\in\IR^{>1}$. Let $Q\in\IZ[X]\setminus\{0\}$ and suppose
  that $Q(\theta)\neq 0$. Then $N_{\op_{Q},0}$ is finite.
\end{lemma}
\begin{proof}
  Assume that $Q(X)=\sum_{i=0}^{d}a_{i}X^{i}$ and let $\op=\op_{Q}$. Let
  $u=Q(\theta)$ and let $\|u\|=\sum_{i=0}^d|a_i|$.  Since
  $N_{\op,0}=N_{-\op,0}$, we may assume that $u>0$.  Choose $\epsilon>0$ such
  that $\epsilon \|u\|<u$ and let $k\in\IN$ be such that for all $n\geq k$ and
  $i\in[d]$, $|r_{n+i}-\theta^ir_n|<\epsilon r_n$. Then, for all
  $n\ge k$, $|a_ir_{n+i}-a_i\theta^ir_n|<\epsilon |a_i| r_n$ (whenever
  $a_i\neq 0$). By our choice of $\epsilon$ we have
  $0< r_n(u-\epsilon\|u\|)< a_0r_n+\cdots+a_mr_{n+d}$, for all $n\ge k$.
\end{proof}
\begin{remark}
  We choose to make an $\epsilon$-style proof of the previous lemma to stress the
  fact that, when $(r_{n+1}/r_{n})$ converges to $\theta$ effectively, we can
  bound effectively the size of $N_{\op,0}$. We can adapt this kind of proof to
  show that the size of
  any set of the form $N_{\op,z}$ can be bounded effectively if finite. This
  will be needed in Section \ref{section-decid}, where we look at the
  decidability of $\str{Z}_{R}$.
\end{remark}
\begin{lemma}
  Suppose that $R$ satisfies a linear recurrence. Then for any $Q\in\IZ[X]$, $N_{\op_{Q},0}$ is either finite or
  $\IN$.  Furthermore, $N_{\op_{Q},0}=\IN$ if and only if $Q(\theta)=0$.
\end{lemma}
\begin{proof}
  Notice that, by assumption, $\op(n)/r_n\to Q(\theta)$. Thus, if
  $Q(\theta)\neq 0$, $N_{\op,0}$ is finite. Otherwise, $P_R$ divides $Q$ and in
  this case, by Proposition \ref{proposition-rec-rel-div}, $N_{\op,0}=\IN$.
\end{proof}
We end this section with by a description of sets of the form $N_{\op,z}$ when
$z\neq 0$.
\begin{proposition}\label{proposition-op-inhom}
  Let $\op$ be an operator and $z\in\IZ\backslash\{0\}$. Then $N_{\op,z}$ is
  finite.
\end{proposition}
\begin{proof}
  Assume that $\op=\op_{Q}$, where $Q(X)=a_{0}+a_{1}X+\cdots+a_{d}X^{d}$. We may
  assume that $Q\neq 0$. Since
  $R$ is regular, we have that $u=\lim_{n\to\infty}\op(n)/r_{n+d}$ is either
  $a_{d}$ (when $\theta=\infty$) or $a_{0}\theta^{-d}+\cdots+a_{d}=\theta^{-d}Q(\theta)$ (when
  $\theta\in\IR^{>1}$). Notice that if $u\neq 0$, then $N_{\op,z}$ is finite
  since $\lim_{n\to\infty}z/r_{n+d}=0$. Now, if $u=0$, then $\theta$ is
  algebraic: $Q(\theta)=0$. Thus, $R$ satisfies a linear recurrence relation,
  and since $P_{R}$ divides $Q$, we have that $N_{\op,0}=\IN$, by Proposition
  \ref{proposition-rec-rel-div}. We conclude that $N_{\op,z}=\emptyset$.
\end{proof}
\subsubsection{Equations and the induced structure}

Let $Q_{1},\ldots,Q_{s}\in\IZ[X]$ be operators and let $\op_{i}=\op_{Q_{i}}$ for
all $i\in[s]$. Let $z\in\IZ$. We consider the
question whether $z$ is in the image of the sum of these operators. This amounts
to determine when the equation $\op_{1}(x_{1})+\cdots+\op_{s}(x_{s})=z$ has a solution in $\IN^{s}$. We call a tuple $\bar{n}\in \IN^{s}$ a \emph{non-degenerate solution}
of $\op_{1}(x_{1})+\cdots+\op_{s}(x_{s})=z$ when the following conditions hold:
\begin{enumerate}
\item $\op_{1}(n_{1})+\cdots+\op_{s}(n_{s})=z$;
\item for all $I\subsetneq[s]$, $\sum_{i\in I}\op_i(n_i)\neq 0$.
\end{enumerate}
We now explain how to decompose $N_{\bar{\op},z}$ into sets of non-degenerate
solutions.\\
\indent Let $\bar{I}=(I_{1},\ldots,I_{k})\in\frak{P}([n])$. To this
partition we associate the following system of equations:
\begin{equation}\label{eq:3}
  \begin{cases}
    \sum_{i\in I_{1}}\op_{i}(n_{i})=z,\\
    \sum_{i\in I_{2}}\op_{i}(n_{i})=0,\\
    \vdots\\
    \sum_{i\in I_{k}}\op_{i}(n_{i})=0.\\
  \end{cases}
\end{equation}
Let \newcommand{\nd}{\mathrm{nd}}
\begin{align*}
  N_{\bar{\op},z,\bar{I}}^{\nd}=\Big\{\bar{n}\in\IN^{s}\,\Big|\,
  & \bar{n}_{I_{1}}
    \text{ is a non-degenerate solution of }
    \sum_{i\in I_{1}}\op_{i}(n_{i})=z
    \text{ and for all } j\in[k]^{>1}\\
  &\bar{n}_{I_{j}}\text{ is a non-degenerate solution of}
    \sum_{i\in I_{j}}\op_{i}(n_{i})=0\Big\}.
\end{align*}
\newcommand{\bop}{\bar{\op}} When $\bar{I}=([s])$, we use $N^{\nd}_{\bop,z}$
instead of $N^{\nd}_{\bop,z,\bar{I}}$. In this setting, we decompose
$N_{\bop,z}$ as
\begin{equation}
  \label{eq:1}
  N_{\bop,z}=\bigcup_{\bar{I}\in\mathfrak{P}([s])}N^{\nd}_{\bop,z,\bar{I}}.
\end{equation}
\indent This decomposition will prove to be quite useful as the set of
non-degenerate solutions of $\op_{1}(x_{1})+\cdots+\op_{s}(x_{s})=z$ is easily
understood. For
instance, Proposition \ref{proposition-sol-eq} implies that for some constant
$m$ depending only on $\bar{\op}$ and $z$, if $\bar{n}$ is a non-degenerate
solution, then
$\max\{|n_{i}-n_{j}|\mid i,j\in[s]\}\le m$.
\begin{proposition}\label{proposition-sol-eq} Let $\op_{1},\ldots,\op_{s}$ be
  operators and $z\in\IZ$. Then, there exist $k\in\IN$ and
  $\bar{m}_1,\ldots,\bar{m}_k$ $\in\IZ^{s}$ such that for all
  $\bar{\ell}\in\IN^{s}$, if $\bar{\ell}\in N_{\bop,z}^{\nd}$ then for some $i\in[k]$,
  $\ell_j=\ell_1+m_{ij}$ for all $j\in [s]$.
\end{proposition}
\begin{proof}
  Assume that $\op_{j}=\op_{Q_{j}}$ where
  $Q_{j}(X)=\sum_{i=0}^{d_{j}}a_{ji}X^{i}$ and $a_{jd_{j}}\neq 0$.  Suppose,
  towards a contradiction, that the proposition is false:
  \begin{enumerate}
  \item[$(\star)$] for all $k\in\IN$ and
    $\bar{m}_{1},\ldots,\bar{m}_{k}\in\IZ^{s}$, there exists
    $\bar{\ell}\in N^{\nd}_{\bop,z}$ such that for all $i\in[k]$,
    $\ell_{j}\neq \ell_{1}+m_{ij}$ for some $j\in[s]$.
  \end{enumerate}
  From this, we construct two sequences $(\bar{\ell}_{i})\subset \IN^{s}$
  and $(\bar{m}_{i})\subset\IZ^{s}$ that will help us reach a contradiction.\\
  \indent Start with any $\bar{\ell}_{1}\in N^{\nd}_{\bop,z}$ and define
  $\bar{m}_{1}$ as $m_{1j}=\ell_{1j}-\ell_{11}$ for all $j\in[s]$. Assuming
  $\bar{\ell}_{i}$ and $\bar{m}_{i}$ are constructed, we let $\bar{\ell}_{i+1}$ be a
  non-degenerate solution obtained from $(\star)$ with $k=i$ and
  $\bar{m}_{1},\ldots,\bar{m}_{i}$. We define $\bar{m}_{i+1}$ as
  $m_{(i+1)j}=\ell_{(i+1)j}-\ell_{(i+1)1}$ for all $j\in[s]$.\\
  \indent We may assume, up to a permutation of $\bar{\op}$ and passing to a
  sub-sequence using the pigeonhole principle, that
  \begin{enumerate}
  \item for all $i\in\IN$, $m_{ij}\le m_{i(j+1)}$ for all $j<s$.
  \end{enumerate}
  Again, using the pigeonhole principle, we may assume that
  \begin{enumerate}
    \setcounter{enumi}{1}
  \item there is $j^{*}\in[s]$ such that for all $i\in\IN$, $m_{ij^{*}}=0$;
  \item\label{item:1} for all $i\in\IN$, $m_{is}<m_{(i+1)s}$.
  \end{enumerate}
  We now decompose the tuples $\bar{m}_{i}$, $i\in\IN$, in two parts according
  to whether the differences $m_{is}-m_{ij}=\ell_{is}-\ell_{ij}$ are bounded. Let
  $J\subset[s]$ be of maximal size such that for all $j\in J$,
  $\max\{m_{is}-m_{ij}\mid i\in\IN\}<\infty$. Notice that $s\in J$ and
  $j^{*}\notin J$.  Applying the pigeonhole principle several times, we may
  assume, without loss of generality, that
  \begin{enumerate}
    \setcounter{enumi}{3}
  \item \label{item:2} for all $j\in J$, there exists $k_j\in\IN$ such that
    $m_{is}-m_{ij}=k_j$ for all $i\in\IN$;
  \item \label{item:3} for all $j\notin J$, $m_{is}-m_{ij}\to\infty$.
  \end{enumerate}

  \indent Let $j_{0}=\min J$ and, for all $i\in\IN$ and $j\in J$, rewrite
  $\ell_{ij}$ as $\ell_{ij_{0}}+(m_{ij}-m_{ij_{0}})=\ell_{ij_{0}}+(k_{j_{0}}-k_{j})$ (note that
  $k_{j_{0}}-k_{j}\ge 0$). Set
  $Q'_{j}(X)=\sum_{n=1}^{d_{j}}a_{jn}X^{k_{j_{0}}-k_{j}+n}$ and
  $\op'_{j}=\op_{Q'_{j}}$. We have, for all $j\in J$,
  \begin{align*}
    \op_{j}(\ell_{ij})&=\op_{j}(\ell_{ij_{0}}+(m_{ij}-m_{ij_{0}}))\\
                   &= \op_{j}(\ell_{ij_{0}}+(k_{j_{0}}-k_{j}))\\
                   &=\op'_{j}(\ell_{ij_{0}}).
  \end{align*}
 Define $Q(X)=\sum_{j\in J}Q'_{j}(X)$, let $d$ be the degree of $Q$ and $a_d$
 be the coefficient of $X^d$ in $Q$. For all $i\in\IN$,
  \begin{align*}
    \sum_{j\in J}\op_{j}(\ell_{ij})&=\sum_{j\in J}\op'_{j}(\ell_{ij_{0}})\\
                                    &=\op_{Q}(\ell_{ij_{0}}).
  \end{align*}
  Notice that by non-degeneracy and the fact that $J$ is a proper (non-empty)
  subset of $[s]$, $Q\neq 0$ (in particular $a_{d}\neq 0$).\\
  \indent Since, for all $n\in\IN$,
  $\lim\limits_{k\to\infty}r_{n}/r_{n+k}=0$, for all $j\notin J$,
  \[u_{j}=\lim_{i\to\infty}\frac{\op_{j}(\ell_{i1}+m_{ij})}{r_{\ell_{1i}+m_{ij_{0}}+d}}=\lim_{i\to\infty}\sum_{n=0}^{d_{j}}\frac{a_{jn}r_{\ell_{i1}+m_{ij}+n}}{r_{\ell_{1i}+m_{ij_{0}}+d}}=0.\]
  (Indeed, by \ref{item:2} and \ref{item:3}, for all $j\notin J$,
  $m_{ij_{0}}-m_{ij}=m_{is}-m_{ij}-k_{j_{0}}\to \infty$.)\\
  \indent Let us now perform a similar calculation. Recall that for all
  $k\in\IN^{>0}$,
  \begin{equation*}
    \lim\limits_{n\to\infty}r_{n}/r_{n+k}=
    \begin{cases}
      \theta^{-k} &\text{if }\theta\in\IR^{>1}\\
      0 &\text{if } \theta=\infty,
    \end{cases}
  \end{equation*}
  So we have that
  \begin{align*}
    u_{J}=
    &\lim_{i\to\infty}\frac{\sum_{j\in J}
      \op_{j}(\ell_{i1}+m_{ij})}{r_{\ell_{i1}+m_{ij_0+d}}}\\
    =&\lim_{i\to\infty}\frac{\op_{Q}(\ell_{ij_{0}})}{r_{\ell_{ij_0+d}}}\\
    =&\begin{cases}
      \theta^{-d}Q(\theta) &\text{if }\theta\in\IR^{>1}\\
      a_{d}&\text{if }\theta=\infty .
    \end{cases}
  \end{align*}
  Thus,
  \[\lim_{i\to\infty}\sum_{j=1}^s\frac{\op_j(\ell_{i1}+m_{ij})}{r_{\ell_{i1}+m_{ij_{0}}+d}}
    =u_{J}+\sum_{j\notin J}u_{j}=u_{J}=\lim_{i\to\infty}\frac{z}{r_{\ell_{i1}+m_{ij_{0}}+d}}=0,\]
  where that last equality comes from the fact that, by \ref{item:1} and
  \ref{item:2}, the sequence $(r_{\ell_{i1}+m_{ij_{0}}+d})$ is not bounded.\\
  \indent Since $u_{J}=0$ and $a_{d}\neq 0$, we must have $\theta\in\IR^{>1}$. In
  particular, $u_{J}=\theta^{d}Q(\theta)$. So $Q(\theta)=0$. Since $R$ is
  regular and $Q\neq 0$, $R$ satisfies a linear recurrence relation and $P_{R}$
  divides $Q$. So by Proposition \ref{proposition-operators}
  $N_{\op_{Q},0}=\IN$, in contradiction with the assumption that
  $\bar{\ell}_{i}$ is non-degenerate for all $i\in\IN$ and the fact that $J$ is
  a proper non-empty subset of $[s]$.
\end{proof}
\indent For an operator $\op_{Q}$, $Q(X)=\sum_{i=0}^{d}a_{i}X^{i}$, let
\[
  N_{\op_{Q},z}^{\circ}=\{n\in N_{\op_{Q},z}\mid \text{for all
  }I\subsetneq[d],\sum_{i\in I}a_{i}r_{n+i}\neq 0\}.
\]
For an $s$-tuple $\bop$ of operators and $\bar n\in\IN^{s}$,  we let
$\op_{\bar n}(\ell)=\sum_{j=1}^{s}\op_{j}(\ell+n_{j})$.

\begin{remark}\label{remark-control}
  Let $\op_{1},\ldots,\op_{s}$ be operators and $z\in\IZ$.
  By Proposition \ref{proposition-op-inhom}, there
  exist $k$ and $\bar{m}_{1},\ldots,\bar{m}_{k}\in \IZ^s$ such that, letting $m_{i}=\min\{m_{i1},\ldots,m_{is}\}$
  and $\bar n_{i}=\bar m_{i}-m_{i}$:
  \begin{quotation}
    \noindent $\bar{\ell}\in N_{\bop,z}^{\nd}$ if and only if for some $i\in[k]$,
    $\ell_{1}+m_{i}\in N_{\op_{\bar n_{i}},z}^{\circ}$ and $\ell_{j}=\ell_{1}+m_{ij}$ for all
    $j\in[s]$.
  \end{quotation}
\end{remark}
\begin{corr}\label{corollary-finiteness}
  Let $\op_{1},\ldots,\op_{s}$ be operators and $z\in\IZ$. Let $k\in\IN$ and
  $\bar{m}_{1},\ldots,\bar{m}_{k}\in\IZ^{s}$ be given by Proposition
  \ref{proposition-sol-eq}. Then $N_{\bop,z}^{\nd}$ is infinite if and only if $z = 0$ and
  $N_{\op_{\bar{m}_{i}},0}^{\circ}$ is infinite for some $i\in[k]$.
\end{corr}
\begin{proof}
  This follows from Propositions \ref{proposition-op-inhom} and
  \ref{proposition-sol-eq}.
\end{proof}
The following corollary states that operators are ultimately injective
functions, unless $N_{\op,0}$ is infinite.
\begin{corr}\label{corollary-inj}
  Let $Q\in\IZ[X]$. Then exactly one of the following holds:
  \begin{itemize}
  \item $N_{\op_{Q},0}=\IN$;
  \item $N_{(\op_{Q},-\op_{Q}),0}\backslash\{(n,n)\mid n\in\IN\}$ is finite.
  \end{itemize}
\end{corr}
\begin{proof}
  Assume $N_{\op_{Q},0}$ is finite. Let us then show that
  $N_{(\op_{Q},-\op_{Q}),0}\backslash\{(n,n)\mid n\in\IN\}$ is finite. Since
  $N_{\op_{Q},0}$ is finite, $N_{(\op_{Q},-\op_{Q}),0}\backslash
  N_{(\op_{Q},-\op_{Q}),0}^{\nd}$ is finite, so we only need to show that
  $N_{(\op_{Q},-\op_{Q}),0}^{\nd}\backslash\{(n,n)\mid n\in\IN\}$ is finite. By
  Proposition \ref{proposition-sol-eq},
  this amounts to show that for all $k\in\IN^{>0}$, the operator
  $\op_{k}(n)=\op_{Q}(n)-\op_{Q}(n+k)$ is such that $N_{\op_{k},0}$ is finite. Assume on
  the contrary that $N_{\op_{k},0}=\IN$ for some $k\in\IN^{>0}$. This implies that $R$ satisfies a
  linear recurrence relation and in that case $P_{R}$ divides
  $Q(X)(1-X^{k})$. But since $\theta>1$, we must have that $P_{R}$ divides $Q$,
  in contradiction with the fact that $N_{\op_{Q},0}$ is finite.
\end{proof}
As a corollary of Proposition \ref{proposition-sol-eq} and the following result,
we obtain the superstability of $R_{\ind}^{0}$.
\begin{proposition}[{\cite[Proposition 5.9]{conant}}]\label{proposition-succ}
  Let $\str{N}$ be the structure $(\IN,S,S^{-1},0)$, where $S(n)=n+1$,
  $S^{-1}(n+1)=n$ and $S^{-1}(0)=0$. Then
  $\str{N}^{1}$ is superstable of $U$-rank $1$.
\end{proposition}
\begin{corr}\label{corollary-ind-struct}
  Let $R$ be a regular sequence. Then $R_{\ind}^{0}$ is definably interpreted in
  $\str{N}$.
\end{corr}
\begin{proof}
  We interpret the domain of $R_{\ind}^{0}$ as $\IN$. Let
  $a_{1},\ldots,a_{s}\in\IZ\backslash\{0\}$. We need to interpret in $\str{N}$
  the set of $s$-tuples of elements in $R$ that satisfy the equation
  $a_1x_1+\cdots+a_sx_s=0$.  For all $i\in[s]$, let $\op_{i}$ be the operator
  $n\mapsto a_{i}r_{n}$. We interpret
  $\{\bar{x}\in R^{s}\mid a_{1}x_{1}+\cdots+a_{n}x_{n}=0\}$ as $N_{\bar{\op},0}$ in
  $\str{N}$. Let us show that $N_{\bar{\op},0}$ is definable in $\str{N}$. As
  explained at the beginning of Section \ref{section-induced}, the set
  \[N_{\bar{\op},0}=\bigcup_{\bar{I}\in\mathfrak{P}([s])}N_{\bar{\op},0,\bar{I}}^{\text{nd}}.\]
  So we need only to show that $N^{\text{nd}}_{\bar{\op},0,\bar{I}}$ is
  definable in $\str{N}$ for all $\bar{I}\in\mathfrak{P}([s])$. We focus on the
  case $\bar{I}=([s])$, the general case being similar. By Remark
  \ref{remark-control}, to show that $N_{\bop,0}$ is definable, we only need to
  show that $N_{\op_{\bar{n}_{i}},0}^{\circ}$ is definable for all
  $i\in[k]$. But as $N_{\op_{\bar{n}_{i}},0}^{\circ}$ is either empty or
  cofinite in $N_{\op_{\bar{n}_{i}},0}$, we only need to show that the latter is
  definable. But
  we know that, by Proposition \ref{proposition-operators}, the set
  $N_{\op_{\bar{n}_{i}},0}$
  is either finite or $\IN$, hence definable.
\end{proof}
\subsection{Every $\L_{R}$-formula is bounded}\label{section-bounded}
Throughout this section, we will use the following notations. Let $\bop$ be a tuple of $k$ operators. Define $\Im{\bop}$ as
\[\{a\in\IZ\mid a=\op_{1}(n_{1})+\cdots+\op_{k}(n_{k})\text{ for some }\bar{n}\in\IN^{k}\}.\]
Notice that $\Im{\bop}=\{a\in\IZ\mid N_{\bop,a}\neq\emptyset\}$. Similarly define $\Imp(\bop)$ as $\Im{\bop}\cap\IN$.\\
\indent This section is devoted to the proof of the following theorem.
\begin{theorem}\label{theorem-cover-z}
  Let $a,d\in\IN$, $d>0$. Then, the set $a+d\IN$ cannot be covered by finitely
  many sets of the form
  $z+\Imp(\bop)$, where
  $\bop$ is a tuple of $k$ operators, $k\in\IN$ and $z\in\IZ$.
\end{theorem}

Recall that a set $A\subset\IN$ is called \emph{piecewise syndetic} if there
exists $d\in\IN^{>0}$ such that for all $k\in\IN$, there exists
$a_{1}<\cdots<a_{k}\in A$ such that $a_{i+1}-a_{i}\le d$ for all $i\in[k-1]$. A
key property of piecewise syndetic sets is the so-called Brown's Lemma.
\begin{theorem}[Brown's Lemma {\cite[Theorem 10.37]{ramsey-integers}}]
  Let $A\subset \IN$ be piecewise syndetic. If $A=A_{1}\cup\cdots\cup A_{n}$,
  then there exists $i\in[n]$ such that $A_{i}$ is piecewise syndetic.
\end{theorem}
In the next proposition, we show that the image of arbitrary linear combinations
of operators is not piecewise syndetic.
\begin{proposition}\label{proposition-image-op}
  Let $\bar{\op}$ be a tuple of $k$ operators. Then $\Imp(\bop)$ is not piecewise
  syndetic.
\end{proposition}
Before giving a proof of Proposition \ref{proposition-image-op}, let us show how
it is used to prove Theorem \ref{theorem-cover-z}.
\begin{proof}[Proof of Theorem \ref{theorem-cover-z}]
  Since $a+d\IN$ is piecewise syndetic, if it were covered by sets of the form
  $z+\Imp(\bop)$, then one of
  them would also be piecewise syndetic, by Brown's Lemma. But this would imply
  that a set of the form
  $\Imp(\bop)$ is piecewise
  syndetic since any translate of a piecewise syndetic set is again piecewise
  syndetic. This contradicts Proposition \ref{proposition-image-op}.
\end{proof}
We will need the following lemma.
\begin{lemma}\label{lemma-synd}
  Let $\op_{1},\ldots,\op_{k}$ be operators and $e\in\IN^{>0}$. Let $X_{e}$ be the
  set $\{a\in\Imp(\bar{\op})\mid \exists a'\in\Imp(\bar{\op}), |a-a'|=e\}$. Then
  there exist a finite set $Z$ of integers such that
  \[X_{e}\subset\bigcup_{z\in
      Z}\bigcup_{I\subsetneq[k]}(z+\Imp(\bar{\op}_{I})).\]
\end{lemma}
\begin{proof}
  We first identify $Z$. Let $a\in X_{e}$. By definition, there is $a'\in\Imp(\bop)$ such that
  $e=a-a'$ or $e=a'-a$, that we shorten by $e=\pm(a-a')$. Since both $a$ and $a'$ are in $\Imp(\bop)$, we can find $\bar{n},\bar{n}'\in\IN^{k}$ such that
  \[a=\sum_{i=1}^{k}\op_{i}(n_{i})\text{  and  }a'=\sum_{i=1}^{k}\op_{i}(n'_{i}).\]
  Since $e=\pm(a-a')$ we can find $I,I'\subset[k]$ such that
  \[e=\pm\left(\sum_{i\in I}\op_{i}(n_{i})-\sum_{i\in I'}\op_{i}(n'_{i})\right),\]
  $(\bar{n}_{I},\bar{n}'_{I'})\in N^{\nd}_{\bop_{I}\cup-\bop_{I'},e}\cup
  N^{\nd}_{-\bop_{I}\cup\bop_{I'},e}$ and
  \[0=\sum_{i\notin I}\op_{i}(n_{i})-\sum_{i\notin I'}\op_{i}(n'_{i}).\]
  We thus let $Z$ be the set
  \[\left\{\sum_{i\in I}\op_{i}(n_{i}),\left.\sum_{i\in I'}\op_{i}(n'_{i})
        \,\right|\, I,I'\subset[k], (\bar{n}_{I},\bar{n}'_{I'})\in
      N^{\nd}_{\bop_{I}\cup-\bop_{I'},e}\cup
      N^{\nd}_{-\bop_{I}\cup\bop_{I'},e}\right\}\cup\{0\}.\]
  By Corollary \ref{corollary-finiteness}, we have that the sets
  $N^{\nd}_{-\bop_{I}\cup \bop_{I'},e}$ and $N^{\nd}_{\bop_{I}\cup-\bop_{I'},e}$
  are finite. Hence $Z$ is finite.\\
  \indent Let us show that
  \[a\in\bigcup_{z\in
      Z}\bigcup_{I\subsetneq[k]}(z+\Imp(\bar{\op}_{I})).\]
  We distinguish three cases.
  \begin{enumerate}
  \item $I=[n]$. In that case, $a\in Z$.
  \item $\emptyset\neq I\subsetneq[n]$. In that case, $[n]\setminus I$ is a
    proper subset of $[n]$ and
    \[a=z+\sum_{i\in [n]\setminus I}\op_{i}(n_{i}),\, z=\sum_{i\in I}\op_{i}(n_{i})\in
      Z.\]
  \item $I=\emptyset$. Since $e>0$, we have that $I'\neq\emptyset$. Now if
    $I'=[n]$, we have $a=0\in Z$. So let us assume that $I'\subsetneq[n]$. In
    that case, $a=\sum_{i\in[n]\setminus I'}\op_{i}(n'_{i})$. Since
    $[n]\setminus I'$ is a proper subset of $[n]$, $a$ has the required form.\qedhere
  \end{enumerate}
\end{proof}
We now prove Proposition \ref{proposition-image-op} by induction on the length of the tuple $\bop$.
\begin{proof}[Proof of Proposition \ref{proposition-image-op}]
  Let $\op$ be an operator. By Lemma \ref{lemma-synd}, we have that $X_{e}$ is
  finite for all $e\in\IN^{>0}$. This implies that $\Imp(\op)$ cannot be piecewise
  syndetic.\\
  \indent Let $k>1$ and assume that the proposition holds for all tuple
  $\bar{\op}$ of length $\le k$. Let $\op_{1},\ldots,\op_{k+1}$ be operators
  such that $\Imp(\bar{\op})$ is infinite. Suppose, towards a contradiction that
  $\Imp(\bar{\op})$ is piecewise syndetic. Assume $d\in\IN^{>0}$ witnesses the
  fact that $\Imp(\bar{\op})$ is piecewise syndetic. Recall that we defined
  $X_{e}$ as $\{a\in\Imp(\bar{\op})\mid \exists a'\in\Imp(\bar{\op}), |a-a'|=e\}$.
  Even though $X_{1}\cup\cdots\cup X_{d}$ may not equal $\Imp(\bar{\op})$, it is
  this subset that will play a key role in the rest of the proof, as it is the
  ``syndetic part of $\Imp(\bar{\op})$ with respect to $d$''. Indeed, the set
  $X_{1}\cup\cdots\cup X_{d}$ is itself piecewise syndetic so that by Brown's
  Lemma, there exists $i\in[d]$ such that $X_{i}$ is also piecewise
  syndetic. But by lemma \ref{lemma-synd} we know that $X_{i}$ is contained in a
  finite union of sets of the form $z+\Imp(\bar{\op}')$, where $\bar{\op}'$ is
  of length $\le k$. But this implies, by Brown's Lemma and the fact that a
  set containing a piecewise syndetic set is itself piecewise syndetic, the existence of a
  piecewise syndetic set of the form $z+\Imp(\bar{\op}')$, where $\bar{\op}'$ is
  of length $\le k$. This contradicts our induction hypothesis. So
  $\Imp(\bar{\op})$ is not piecewise syndetic, which is what we wanted.
\end{proof}
\begin{corr}\label{corollary-bounded}
  Let $R$ be enumerated by regular sequence. Then $R$ is bounded.
\end{corr}
\begin{proof}
  The proof follows \cite[Lemma 3.4 and Lemma 3.5]{Pal-Sk} and is based on \cite[Proposition 2.1]{Cas-Zieg}. That proposition states, for an $\L$-structure $\str{M}$ and a small subset $A$,
  that $A$ is bounded if $\str{M}$ is both stable and nfcp over $A$.
  The proof of
  \cite[Proposition  2.1]{Cas-Zieg} is an induction on the number of quantifiers in
  $\L_{A}$-formulas.
  Palac\'{i}n and Sklinos \cite[Lemma 3.5]{Pal-Sk} noticed that the smallness of $A$ and nfcp over $A$ can be weakened to the following statement:
  \begin{quote}
    $(\star)$  for any $\L$-formula $\phi(\bar{x},y,\bar{z})$,
    there exists $k\in\IN$ such that
    \[\str{M}_{A}\models\forall\bar{x}\left(\left(\forall\bar{z}_{0}\in
          A\ldots\forall\bar{z}_{k}\in A\exists y\bigwedge_{j<
            k}\phi(\bar{x},y,\bar{z}_{j}) \right)\rightarrow\exists
        y\forall\bar{z}\in A\phi(\bar{x},y,\bar{z})\right).\]
  \end{quote}
  Let us now explain how one proceeds in the case of a regular sequence.\\
  \indent First let us show that, for any $\L_{g}$-formula
    $\phi(\bar{x},y,\bar{z})$, any consistent set the form
\[\Gamma(y)=\{\phi(\bar{b},y,\bar{\alpha})\mid \bar{\alpha}\in R^{n}\},\]
where
    $\bar{b}\in\IZ$ and $n$ is the length of the tuple $\bar{z}$, is realized by
    some $c\in\IZ$.
    \par Using quantifier elimination, the consistency of $\Gamma(y)$
    and the properties of the congruence relations, we can assume that $\phi$ is
    of the form
    \[\bigwedge_{i\in I_{1}}t_{i}(\bar{x},y,\bar{z})=0
      \wedge\bigwedge_{i\in I_{2}}t_{i}(\bar{x},y,\bar{z})\neq 0
      \wedge\bigwedge_{i\in I_{3}}D_{n_{i}}(y+t'_{i}(\bar{x},\bar{z})),
    \]
    where $t_{i}(\bar{x},y,\bar{z})$ and $t'_{j}(\bar{x},\bar{z})$ are terms for
    all $i\in I_{1}\cup I_{2}$ and $j\in I_{3}$. We may further assume that
    $I_{1}=\emptyset$ (otherwise it is clear that $y\in \IZ$).  Given
    $i\in I_{2}$, the term $t_{i}(\bar{b},y,\bar{z})$ is equal to
    $m_{i}y+z_{i}+a_{1}z_{1}+\cdots+a_{\ell_{i}}z_{\ell_{i}}$, where
    $m_{i},z\in\IZ$, $\bar{a}_{i}\in\IZ^{\ell_{i}}$ and $\ell_{i}\in\IN$.
    Notice that we may assume that $m_{i}=m$ for all $i\in I_{2}$ (otherwise, we
    multiply the inequation $t_{i}(\bar{x},y,\bar{z})\neq 0$ by $\prod_{j\neq
      i}m_{j}$ for all $i\in I_{2}$). \\
    \indent Thus $\Gamma(y)$ expresses the fact that $y$ is
    in a coset of a subgroup of $\IZ$, say $c+d\IZ$ for some $c,d\in\IN$,
    and $my$ is not in the set
    \[
      X=\bigcup_{i\in I_{2}}
      \{z_{i}+a_{1}z_{1}+\cdots+a_{\ell_{i}}z_{\ell_{i}}\mid \bar{z}\in R^{\ell}\}.
    \]
    But, by Theorem \ref{theorem-cover-z}, $X$ does not
    cover $mc+md\IN$ (use operators of the form $n\mapsto ar_{n}$ to apply the
    theorem). So there is $s\in\IN$ such that $m(c+ds)$ is not in $X$,
    which is what we wanted. This shows \cite[Lemma 3.4]{Pal-Sk} for regular
    sequences.\\
    \indent From this and the fact that $\str{Z}$ has nfcp, we deduce directly that $(\star)$
    holds for regular sequences. Thus $R$ is bounded.
  \end{proof}
  \begin{remark}
    G. Conant and C. Laskowski recently showed in \cite[Theorem 2.8]{conant3} that
    \emph{any} subset of a weakly minimal group is bounded, at the cost of
    adding constants in the language\footnote{A group $G$
      is weakly minimal when its $\{+,0\}$-theory is superstable of $U$-rank $1$.}. This applies in
    particular to $\str{Z}$ and thus eliminates the necessity of Corollary
    \ref{corollary-bounded} in the proof of our Theorem
    \ref{theorem-sup}. However, the material of this section -- specifically
    Proposition \ref{proposition-image-op} -- is needed in
    Section \ref{section-axiomatization} towards the proof of Theorem
    \ref{thm-eq}, namely in Propositions \ref{proposition-op-nstd} and
    \ref{prop-mc-1}.
  \end{remark}
\subsection{Main theorem}
We are now able to prove the main theorem of this section.
\theoremstyle{theorem}
\newtheorem*{theorem-sup}{Theorem \ref{theorem-sup}}
\begin{theorem-sup}
  Let $R$ be enumerated by a regular sequence.  Then $\Th(\str{Z}_{R})$ is superstable of
  Lascar rank $\omega$.
\end{theorem-sup}
\begin{proof}
  By Proposition \ref{proposition-succ} and Corollary
  \ref{corollary-ind-struct}, we get that $R_{\textrm{ind}}^{0}$ is
  superstable. Furthermore, by Corollary \ref{corollary-bounded} $R$ is
  bounded. So, we deduce from Theorem \ref{thm-casanovas} that $\str{Z}_R$ is
  superstable.  Since $\str{Z}_{R}$ is a proper expansion of $\str{Z}$, it must
  have Lascar rank $\ge\omega$ by \cite[Theorem 1]{Pal-Sk}. So what remains to
  be shown is that the rank is
  $\le\omega$.
  For $R=\Pi_{q}$, this is done in \cite[Theorem 2]{Pal-Sk}  and the only property that we need to check here is that the
  $U$-rank of $R_{\ind}$ is $1$. But, by Proposition \ref{proposition-succ}, $\str{N}$ has $U$-rank 1 and by Corollary \ref{corollary-ind-struct},
  $R_{\ind}$ is definably
  interpreted in $\str{N}$.

 For convenience of the reader, we give now a sketch of \cite[Theorem 2]{Pal-Sk}.
 Let $\str{G}\succ\str{Z}_{R}$ be a monster model. Since $\str{G}$ is a superstable
  group, it has a unique generic type $p\in S_{1}^{\L_{R}}(\emptyset)$ in the
  connected component of $\str{G}$.
  Since the type $p$ has maximal $U$-rank (in the
  sense of the $\L_{R}$-theory of $\str{Z}_{R}$), it suffices to show that $U(p)\le\omega$.
  By definition, this amounts to show that any forking extension of $p$ has finite
  $U$-rank.

  First let us show that if $u\in\acl[\L_R]{R(G), B}$, then $U(u/B)<\omega$ $(\dagger)$.
  Let $\bar{c}\in R(G)^{n}$ be such that $u\in\acl[\L_R]{R(G), \bar c}$.
 By Proposition \ref{proposition-succ} and
  Corollary \ref{corollary-ind-struct}, $U(R(G))\le n$. Then by Lascar's inequality $U(u/B)\le
  U(\bar{c}/B)\le n<\omega$.

  Now consider a forking extension of $p$, say $\tp[\L_{R}]{b/B}$ with
   $|B|<|G|$ and suppose that $b\notin\acl[\L_R]{R(G),B}$.

  Let $q=\tp{a/B}$ be a non-forking extension of
  $p$, $U(q)\ge \omega$. By $(\dagger)$ $a\notin
  \acl[\L_R]{R(G),B}$ and hence $a\notin\acl[\L_g]{R(G),B}$. So, in the
  $\L_{g}$-theory of $\str{Z}$, $U(a/R(G),B)=1$. Thus, $\tp[\L_{g}]{a/R(G),B}$
  is a generic type. This is also true for $\tp[\L_{g}]{b/R(G),B}$. So there
  exists $g\in G$ such that
  $\tp[\L_{g}]{b+g/R(G),B}=\tp[\L_{g}]{a/R(G),B}$. This implies, by
  \cite[Corollary 3.7]{Pal-Sk}, $\tp[\L_{R}]{b+g/B}=\tp[\L_{R}]{a/B}$. So
  $\tp[\L_{R}]{b/B}$ is also generic and hence non-forking, a contradiction.
  \end{proof}

As we mentioned in the introduction, there is an overlap between Theorem
\ref{theorem-sup} and \cite[Theorem 7.1]{conant}. We recall that a sequence
$(r_{n})$ is  \emph{geometrically sparse} \cite[Definition 6.2]{conant} if there
exists a sequence $(\lambda_{n})\subset\IR^{\ge 1}$ such that
$X=\{\lambda_{m}/\lambda_{n}\mid n<m\}$ is closed and discrete and
$\sup_{n\in\IN}{|r_{n}-\lambda_{n}|}<\infty$.\\
\indent In our comparison between our theorem and \cite[Theorem 7.1]{conant}, we will
use the following observation about sparse sequences $(r_{n})$ such that
$r_{n+1}/r_{n}\to\theta\in \IR^{>1}$: for such a sequence, there exists
$\tau\in\IR^{\ge 1}$ such that $r_{n}/\theta^{n}\to\tau$. Indeed,
since $\sup_{n\in\IN}{|r_{n}-\lambda_{n}|}<\infty$ and $r_{n+1}/r_{n}\to\theta$,
we have
  \[\lim_{n\to\infty}\frac{\lambda_{n}}{r_{n}}=1\; {\rm and}\;
    \lim_{n\to\infty}\frac{\lambda_{n+1}}{\lambda_{n}}=\theta.\]
But, as $X$ is closed and discrete, if the sequence $(\lambda_{n+1}/\lambda_{n})$ converges, then it is
ultimately constant and so ultimately equal to $\theta$. Hence
$(\lambda_{n}/\theta^{n})$ converges to some limit $\tau\in\IR^{\geq 1}$.\\
\indent Now, let us discuss the overlap between  Theorem
\ref{theorem-sup} and \cite[Theorem 7.1]{conant}:
\begin{itemize}
\item the case where $r_{n+1}/r_n\to\infty$ is completely covered by
  \cite[Theorem 7.1]{conant} (as a consequence of \cite[Proposition
  6.3]{conant});
\item the case where $r_{n+1}/r_{n}\to\theta$ and $\theta$ is algebraic, we
  provide more examples of expansions by recurrence relations than \cite[Theorem
  7.1]{conant}: in addition to our hypotheses, $\theta$ needs to be either a
  Pisot number or a Salem number in order to be geometrically sparse. In fact,
  we can show by direct calculations that the sequence defined by
  $r_{n+2}=5r_{n+1}+7r_{n}$, $r_{1}=1$ and $r_{0}=0$, is regular but not
  geometrically sparse. Indeed, first notice that for all $n\in\IN$,
  $r_{n}=\alpha(\lambda_{+}^{n}-\lambda_{-}^{n})$, where $\alpha=1/\sqrt{53}$,
  $\lambda_{\pm}=(5\pm\sqrt{53})/2$. Then assume that there is a sequence
  $(\lambda_{n})$ such that $\sup\{|r_{n}-\lambda_{n}|\mid n\in\IN\}=k\in\IR$. Let
  $\kappa_{n}=\lambda_{n}/\lambda^{n}_{+}$ and note that
  $\kappa_{n}\to\alpha$. Now we have
  $|r_{n}-\lambda_{n}|=|(\alpha-\kappa_{n})\lambda^{n}_{+}-\alpha\lambda^{n}_{-}|$
  $(*)$. We want to show that $\{\lambda_{m}/\lambda_{n}\mid n\le m\}$ cannot be
  both closed and discrete. Assume towards a contradiction that
  $\{\lambda_{m}/\lambda_{n}\mid n\le m\}$ is closed and discrete. In that case,
  $\lambda_{n+1}/\lambda_{n}$ is ultimately equal to $\lambda_{+}$. Thus
  $\kappa_{n+1}/\kappa_{n}$ is ultimately equal to 1. This in turn implies that
  $\kappa_{n}=\alpha$ for all sufficiently large $n\in\IN$. But in this case,
  $(*)=|\alpha\lambda^{n}_{-}|$ for all sufficiently large $n\in\IN$, in
  contradiction with the boundness of $(*)$;
\item for the case where $r_{n+1}/r_{n}\to\theta$ and $\theta$ is
  transcendental, the overlap is less precise and we did not manage make a clear
  distinction between the two results. However, if $(r_n)$ is
  geometrically sparse,
  that is $\sup_{n\in\IN}{|r_{n}-\lambda_{n}|}<\infty$ for some sequence
  $(\lambda_{n})\subset\IR^{\ge 1}$ such that
  $X=\{\lambda_{m}/\lambda_{n}\mid n<m\}$ is closed and discrete, then the sequence
  $(r_{n}+n)$ is not geometrically sparse but satisfies Theorem
  \ref{theorem-sup}. Since there exists $\tau\in\IR^{\ge 1}$ such that
  $r_{n}/\theta^{n}\to\tau$, we may assume
  $\lambda_{n}=\tau\theta^{n}$. \\
  Assume, towards a contradiction, that there is a sequence $(\lambda'_{n})$
  such
  that $\sup_{n\in\IN}|r_{n}+n-\lambda'_{n}|<\infty$ and
  $X'=\{\lambda'_{m}/\lambda'_{n}\mid n\le m\}$ is closed in discrete. Now let
  $\kappa_{n}=\tau\theta^{n}+n-\lambda'_{n}$. Notice that $(\kappa_{n})$ is
  bounded since we assumed $(r_{n})$ geometrically sparse. So we have that
  $\lambda'_{n+1}/\lambda'_{n}\to\theta$. Since
  $X'$ is closed in discrete, this last sequence
  is ultimately constant: for all $n\in\IN$ sufficiently large,
  \[\theta
    =\frac{\tau\theta^{n+1}+n+1-\kappa_{n+1}}{\tau\theta^{n}+n-\kappa_{n}}.\]
  So, for all sufficiently large $n\in\IN$,
  $n(\theta-1)=1-\kappa_{n+1}+\theta\kappa_{n}$, a contradiction.
\end{itemize}
The assumption on $\theta$ when it is algebraic cannot be removed. Indeed for
all $a,b\in\IN$ with $b>0$, if $R$ is enumerated by $(a+bn)$, then $\str{Z}_{R}$
is unstable\footnote{This is also true for any sequence $(r_{n})$ such that
  there exists $k\in\IN$ such that for all $n\in\IN$, $|r_{n+1}-r_{n}|\le k$.}
and satisfy the linear recurrence $r_{n+2}=2r_{n+1}-r_{n}$.

\section{The theory $T_{R}$}\label{section-axiomatization}
\renewcommand{\Im}{\mathrm{Im}}

In this section, we axiomatize, in a language $\L\supset\L_{g}$, the theory
$T_{R}$ of $\str{Z}_{R}=(\IZ,+,-,0,R)$, where $R(\IZ)$ is enumerated by a
regular sequence $(r_{n})$. We show that $T_{R}$ has quantifier elimination in
$\L$ and has a
prime model (and hence $T_{R}$ is complete). Using this quantifier elimination
result, we then
prove, by means of counting of types, that $T_{R}$ is superstable. As a
consequence we deduce that the $\L_{R}$-theory $\Th(\str{Z}_{R})$ is
superstable. We then close this section with a decidability result and point out
the parallel between our results and those corresponding to expansions of
Presburger arithmetic. We furthermore use a quantifier elimination result of
F. Point to show that expansions of Presburger arithmetic by regular sequences
are NIP.\\
\indent From now on, we fix a infinite set $R(\IZ)\subset\IN$ that is enumerated
by a regular sequence. We know then from the previous section that
\begin{enumerate}
\item \label{bop-cst}(Proposition \ref{proposition-sol-eq}) for all
  $Q_{1},\ldots,Q_{s}\in\IZ[X]$, there is $k=k(\bar{Q})\in\IN$ and a finite set
  $E=E_{\bar{Q}}\subset\IZ^{k}$ such that for all $\bar{\ell}\in\IN^{s}$, if
  $\bar{\ell}\in N_{\bop,0}^{\nd}$ then for some $\bar{m}\in E$,
  $l_{i}=l_{1}+m_{i}$ for all $i\in[k]$, where $\op_{i}=\op_{Q_{i}}$ for all
  $i\in[s]$;
\item \label{op-cst} (Corollary \ref{corollary-inj}) for all $Q\in\IZ[X]$,
  either $N_{\op_{Q},0}=\IN$ or there is $e=e(Q)$ such that for all $n,m>e$,
  if $\op_{Q}(n)=\op_{Q}(m)$ then $n=m$. Let $\triv$ be the
  set of $Q\in\IZ[X]$ such that $N_{\op_{Q},0}=\IN$. Note that $\triv=\{0\}$
  unless $\theta$ is algebraic, in which case $\triv$ is the ideal of $\IZ[X]$
  generated by $P_{R}$.
\end{enumerate}
Our choice of $\L$ will allow us to express the two above properties in a first
order way.
\subsection{Axiomatization and quantifier elimination}\label{sec:axio}
Let us define the language in which we axiomatize $\str{Z}_{R}$. As mentioned in
the introduction, $\L_{g}$ is the language $\{+,-,0,D_{n}\mid n\in\IN^{>1}\}$ and
$\L_{S}$ is the language $\{S,S^{-1},c\}$. These new symbols are interpreted in
$\str{Z}$ as follows: $c$ is interpreted as $r_{0}$, for all $n\in\IN$, $S(r_{n})=r_{n+1}$,
$S^{-1}(r_{n+1})=r_{n}$, $S^{-1}(r_{0})=r_{0}$ and $S(z)=z=S^{-1}(z)$ for all
$z\in\IZ\setminus R(\IZ)$. To each $Q\in\IZ[X]$, we let
$\op=\op_{Q}$ be the $\L_{g}\cup\L_{S}$-term $\sum_{i=0}^{d}n_{i}S^{i}(x)$
($\star$), where $Q(X)=\sum_{i=0}^{d}n_{i}X^{i}$ and $S^{0}(x)=x$. Notice that
such terms are similar to the operators of the previous section: in fact a term
of the form ($\star$) composed with the function $n\mapsto r_{n}$ will be an
operator in the sense of Definition \ref{definition-op}. This explains why we
decided to keep the same notations. Furthermore, in this section, the symbol
$\op$ will always denote a term of the form $\op_{Q}$.\\
\indent We now work in $\L_{g}\cup\{1\}\cup\L_{S}$, where $1$ is a constant
symbol that is interpreted in $\str{Z}_{R}$ by the integer $1$. For $n,m\in\IN$ we
let $\IZ[X]^{n\times m}$ be the set of $n\times m$ matrices with
entries in $\IZ[X]$.\\
\indent Let $[Q]=(Q_{ij})\in\IZ[X]^{n\times m}$ and let
$\phi_{[Q]}(\bar{x},\bar{y})$ be the formula
\[\bigwedge_{i\in[n]}\sum_{j\in[m]}\op_{Q_{ij}}(x_{j})=y_{i}.\]
Notice that, working in $\str{Z}_{R}$, the formula
$\exists \bar{x}\in R\;\phi_{[Q]}(\bar{x},\bar{y})$ expresses the fact that
\[\bigcap_{i\in[n]}N_{\bop_{i},y_{i}}\neq\emptyset.\]
Let $D=\{(P_{i},\ell_{i},k_{i})\mid i\in[m]\}$ be a (finite) subset of
$\IZ[X]\times\IN\times\IN$ such that if $(P,\ell,k)\in D$, then $k<\ell$. We
call such a set $D$ a \emph{set of divisibility conditions} and define
$\phi_{D}(\bar{x})$ as the formula
\[\bigwedge_{i\in[m]}D_{\ell_{i}}(\op_{P_{i}}(x_{i})+k_{i}).\]
\indent To $[Q]$ and $D$ as above, we associate an $n$-ary predicate
$\Im_{[Q],D}(\bar y)$. When $D$ is empty, we write $\Im_{[Q]}$ instead of
$\Im_{[Q],D}$. In $\str{Z}_{R}$, the predicate $\Im_{[Q],D}(\bar y)$ is
interpreted as
follows: for all $z\in\IZ^{|\bar y|}$, $\Im_{[Q],D}(\bar z)$ if and only if
$\exists \bar{x}\in R(\phi_{[Q]}(\bar{x},\bar{y})\wedge\phi_{D}(\bar{x}))$. So
this symbol states that $\bar{z}$ is in the image of sums of operators, and a
witness of this fact satisfies certain divisibility conditions.\\
\indent Finally let $\L$ be the language
\[\L_{g}\cup\{1\}\cup\L_{S}\cup\{R\} \cup\{\Im_{[Q],D}\mid [Q]\text{ and }D\textrm{ as
    above}\},\]
and we let $\str{Z}_{R,\L}$ be the $\L$-expansion of $\str{Z}_{R}$ described above.\\
\indent We fix an axiomatization $T_{1}$ of $\Th(\IZ,+,-,0,1,D_{n}\mid 1<n\in\IN)$ (see
\cite[Chapter 15, Section 15.1]{rothmaler}) and we let $T_{2}$ be the following
universal axiomatization of $\Th(R,S,S^{-1},c)$:
\[T_{2}=\{\forall x(x\neq c\rightarrow S(S^{-1}(x))=x), \forall
  x(S^{-1}(S(x))=x), \forall x(S(x)\neq c), S^{-1}(c)=c\}.\]
We will denote by $T_{2}^{R}$ the theory obtained by relativizing to the
predicate $R$ the quantifiers appearing in each element of $T_{2}$.  We will
frequently use the fact that, modulo $T_{1}$, a formula of the form
$\neg D_{n}(x)$ is equivalent to
\[\bigvee_{k=1}^{n-1} D_{n}(x+k).\]
Let $\str{M}$ be an $\L$-structure. Let $Q_{1},\ldots,Q_{n}\in\IZ[X]$. We say
that $\bar{a}\in M^n$ is a non-degenerate solution of
$\sum_{i=1}^{n}\op_{Q_{i}}(x_{i})=y$ if it is a solution of this equation and
no proper sub-sum is equal to $0$. This
can be expressed by the following first-order formula $N_{\bar{Q}}^{\nd}(\bar{x},y)$
\[\sum_{i=1}^{n}\op_{Q_{i}}(x_{i})
  =y\wedge \bigwedge_{I\subsetneq[n]}\sum_{i\in
    I}\op_{Q_{i}}(x_{i})\neq 0.\]
Let $T_{R}$ be the following set of axioms.
\begin{enumerate}[label=(Ax.\arabic*)]
\item\label{axiom-group} $T_{1}$;
\item\label{axiom-seq} $T_{2}^R$;
\item\label{axiom-r0} $c=r_{0}$ (that is $c$ equals the term
  $\underbrace{1+\cdots+1}_{r_{0}\textrm{ times}}$);
\item\label{axiom-succ} $\forall x(\neg R(x)\rightarrow S(x)=x)$;
\item\label{axiom-cong} for all $\ell_{1},\ldots,\ell_{n}$, all $0\le
  k_{i}\le\ell_{i}$ and $\bar{Q}\in\IZ[X]^{n}$, if
  \[
    \left\{\bar{z}\in R(\IZ)^{n}
    \,\left|\,\str{Z}_{\L,R}\models\bigwedge_{i\in[n]}\right.
    D_{\ell_{i}}(\op_{Q_{i}}(z_{i})+k_{i})\right\}
  =\{\bar{w}_{1},\ldots,\bar{w}_{m}\}
  \]
  then we add the axiom
  \[
    \forall\bar{x}\in R\left(\bigwedge_{i\in[n]}
      D_{\ell_{i}}
      (\op_{Q_{i}}(x_{i})+k_{i})\rightarrow
      \bigvee_{i\in[m]}\bigwedge_{j\in[n]}
    x_{i}=w_{ij}\right);
  \]
\item\label{axiom-image} for all $[Q],D$ as above, we add the axiom
  \[\forall
    \bar{y}\left(\Im_{[Q],D}(\bar{y})\leftrightarrow\exists\bar{x}\in
      R(\phi_{D}(\bar{x}) \wedge\phi_{[Q]}(\bar{x},\bar{y}))\right);\]
\item\label{axiom-operator} for all $Q\in\triv$, we add the axiom
  \[\forall x\in R\;\;\op_{Q}(x)=0\]
  and for all $Q\notin\triv$ we add
  \[ \forall x,y\in R\left(x>e\wedge y>e\wedge x\neq y\rightarrow \op_{Q}(x)\neq\op_{Q}(y)\right),\]
  where $e=e(Q)$ (see \ref{op-cst} on page~\pageref{op-cst});
\item\label{axiom-eq} for every $Q_{1},\ldots,Q_{s}\in\IZ[X]$, we add the axiom
  \begin{align*}
    \forall \bar{x}\in R\left(N_{\bar{Q}}^{\nd}(\bar{x},0)\rightarrow
    \bigvee_{\bar{m}\in E}\bigwedge_{i\in[s]}x_{i}=S^{m_{i}}(x_{1})\right),
  \end{align*}
  where $E=E_{\bar{Q}}$ (see \ref{bop-cst} on page~\pageref{bop-cst}).
\end{enumerate}

Note that  $\str{Z}_{R,\L}$ is a model of $T_{R}$. Indeed, \ref{axiom-operator}
follows from Corollary \ref{corollary-inj} and \ref{axiom-eq} follows from
Proposition \ref{proposition-sol-eq}. In particular $T_R$ is consistent.
Also note that all axioms but the defining axioms for the congruences $D_n$
and the predicates $\Im_{[Q],D}$ are universal.\\
\indent The main result of this section is the following
theorem.
\begin{theorem}\label{thm-eq}
  The $\L$-theory $T_{R}$ has quantifier elimination.
\end{theorem}
Notice that the sequence $R=\{2^{n}+n\mid n\in\IN\}$ does not satisfy
\ref{axiom-eq}:
considering the term $\op(x)=S^{2}(x)-3S(x)+2x$, one can find infinitely many
(non-degenerate) solutions of the equation $\op(x_{1})-\op(x_{2})=0$. In view of
Theorem \ref{theorem-sup2} this is not surprising because the structure
$(\IZ,+,0,1,R,S)$ is known to be unstable: $\IN$ is definable by the formula
$\exists y\in R\ y\neq 1\wedge (2y-S(y)=x)$. However, we do not know if there
exists a sequence $R$ such that $\str{Z}_{R}$ is (super)stable and
$(\IZ,+,0,1,R,S)$ unstable.\\
\indent To establish quantifier elimination, we use the following
criterion. Given two models $\str{M}_{0}\subset\str{M}$ of an arbitrary theory, we say that
$\str{M}_{0}$ is \emph{1-e.c.} in $\str{M}$ if any quantifier-free definable subset of
$M$, defined with parameters in $M_{0}$, has a non-empty intersection with
$M_{0}$.
\begin{proposition}[{\cite[Corollary 3.1.12]{marker}}]\label{prop-qe-crit}
  Let $T$ be an $\mathcal{L}$-theory such that
  \begin{enumerate}
  \item ($T$ has algebraically prime models) for all $\str{M}\models T$ and all
    $\str{A}\subset\str{M}$, there exists a model $\overline{\str{A}}$ of $T$
    such that for all $\str{N}\models T$, any embedding $f:\str{A}\to\str{N}$
    extends to an embedding $\bar{f}:\overline{\str{A}}\to\str{N}$;
  \item ($T$ is 1-e.c.) for all $\str{M}_{0},\str{M}\models T$, if
    $\str{M}_{0}\subset\str{M}$ then $\str{M}_{0}$ is 1-e.c. in$\str{M}$.
  \end{enumerate}
  Then $T$ has quantifier elimination.
\end{proposition}
\indent The proof of Theorem \ref{thm-eq} will be a
consequence of Proposition \ref{prop-qe-crit} and the work done in the following
subsections. In Section
\ref{subsection-cons-t}, we prove
several direct consequences of $T_{R}$ regarding equations of the form
$\op_{1}(x_{1})+\cdots+\op_{n}(x_{n})=a$. In Section
\ref{section-alg-prime}, we give a detailed construction of algebraically prime
models of $T_{R}$. Finally, we show in Section \ref{section-1ec} that $T_{R}$ is
1-e.c.
\begin{corr}\label{cor-eq}
   The $\L$-structure $\str{Z}_{R,\L}$ is a prime model of $T_{R}$. In
  particular $T_{R}$ is complete.
\end{corr}
\begin{proof}
  Since $\str{Z}_{R,\L}$ is an algebraically prime model and $T_{R}$ has
  quantifier elimination, $\str{Z}_{R,\L}$ is a prime model. Therefore, $T_{R}$ is
  complete.
\end{proof}
\subsubsection{Equations in $T_{R}$}\label{subsection-cons-t}
\begin{definition} A term $\op$ is said to be \emph{trivial} if $\op=\op_{Q}$
  for some $Q\in\triv$.
\end{definition}

\begin{definition}
  Let $\str{M}\models T_R$ and $a,b\in R$. The orbit of $a$ is the set
  $\{S^{k}(a)\mid k\in\IZ\}$ and is denoted by $\Orb(a)$. We say that $a$ and $b$
  are in the same orbit if and only if $b\in\Orb(a)$.
\end{definition}
The relation ``$a$ and $b$ are in the same orbit'' is an equivalence relation.
\begin{lemma}\label{lemma-sol-eq-orbits}
  Let $\str{M}\models T_{R}$. Let $\bar{\op}$ be a $n$-tuple of non-trivial
  terms, $n>1$, and let $b_{1},\ldots,b_{k}\in R$, $1<k\le n$, be in different
  orbits.
  \begin{enumerate}
  \item If $k>n/2$, then for all $c_{k+1},\ldots,c_{n}\in R$,
    \[\sum_{i=1}^{k}\op_{i}(b_{i})+\sum_{i=k+1}^{n}\op_{i}(c_{i})\neq 0;\]
  \item If $k\le n/2$, then for all $c_{k+1},\ldots,c_{n}\in R$, the elements
    $b_{1},\ldots,b_{k},c_{k+1},\ldots,c_{n}$ do not form a non-degenerate
    solution of the equation $\sum_{i=1}^{n}\op_{i}(x_{i})=0$. Moreover, if
    $\sum_{i=1}^{k}\op_{i}(b_{i})+\sum_{i=k+1}^{n}\op_{i}(c_{i})=0$, then for
    all $i\in[k]$ there exists a non-empty $P_{i}\subset\{k+1,\ldots,n\}$ such
    that $P_{i}\cap P_{i'}=\emptyset$ for all $i\neq i'\in[k]$ and for all
    $i\in[k]$ $b_{i},(c_{j})_{j\in P_{i}}$ is a non-degenerate solution of
    \[\op_{i}(x_{i})+\sum_{j\in P_{i}}\op_{j}(x_{j})=0.\]
  \end{enumerate}
\end{lemma}
\begin{proof}
  Let $c_{k+1},\ldots,c_{n}\in R$. It is clear from \ref{axiom-eq} that
  $b_{1},\ldots,b_{k},c_{k+1},\ldots,c_{n}$ cannot be a non-degenerate solution
  of \[\sum_{i=1}^{n}\op_{i}(x_{i})=0,\]
  since, for instance, $b_{1}$ is not in the same orbit as $b_{2}$.  Suppose
  $b_{1},\ldots,b_{k},c_{k+1},\ldots,c_{n}$ is degenerate. Then one shows by
  induction on $n$ that there exists a partition $(P_{1},\ldots,P_{\ell})$ of
  $[n]$ such that for all $j\in[\ell]$
  $(b_{i})_{i\in P_{j}\cap[k] },(c_{i})_{i\in P_{j}\cap\{k+1,\ldots,n\}}$ is a
  non-degenerate solution of
  \[\sum_{i\in P_{j}\cap[k]}\op_{i}(x_{i})+\sum_{i\in
      P_{j}\cap\{k+1,\ldots,n\}}\op_{i}(x_{i})=0.\]
  Since $b_{1},\ldots,b_{k}$ are in different orbits, we must have, by
  \ref{axiom-eq}, $|P_{j}\cap[k]|\le 1$ for all $j\in[\ell]$. Also, since all
  terms involved are non-trivial, we must have
  $|P_{j}\cap\{k+1,\ldots,n\}|>0$ for all $j\in[\ell]$. This implies in
  particular that $k\le n/2$ and finishes the proof of the lemma.
\end{proof}
We now show that \ref{axiom-eq} is true for non-homogeneous equations.
\begin{proposition}\label{prop-sol-inhom}
  Let $\str{M}\models T_{R}$, $\bar{Q}\in \IZ[X]^{n}$ and $a\in M$, $a\neq 0$. Then there exist
  $\bar{b}_{1},\ldots,\bar{b}_{k}\in R$ such that
  \[\str{M}\models\forall\bar{x}\in R\left(N_{\bar{Q}}^{\nd}(\bar{x},a)\rightarrow\bigvee_{j=1}^{k}
      \bigwedge_{i=1}^{n}x_{i}=b_{ji}\right).\]
\end{proposition}
\begin{proof}
  Let $\op_{i}=\op_{Q_{i}}$ for all $i\in[n]$. Assume there
  exist infinitely many distinct non-degenerate solutions $\bar{b}_{i}\in
  M^{n}$, $i\in\IN$, of the equation
  \[\op_{1}(x_{1})+\cdots+\op_{n}(x_{n})=a.\]
  We will reach a contradiction using \ref{axiom-eq} applied to the equation
  \[\sum_{i=1}^{n}\op_{i}(x_{i})-\sum_{i=n+1}^{2n}\op_{i-n}(x_{i})=0,\]
  which we denote by $\phi(\bar{x})$.\\
  \indent We have that for all $i\in\IN$, the tuple
  $(\bar{b}_{0},\bar{b}_{i})$ is a solution of
  $\phi(\bar{x})$. We may assume that there exists
  a partition $I=(I_{1},\ldots,I_{\ell})$ of $[2n]$ such that for all $i\in\IN$
  and all $j\in[\ell]$, $(b_{0k}\mid k\in I_{j}\cap[n])$,
  $(b_{ik}\mid k+n\in I_{j}\setminus[n])$ is a non-degenerate solution of the equation
  \[\sum_{i\in I_{j}\cap[n]}\op_{i}(x_{i})-
    \sum_{i\in I_{j}\setminus[n]}\op_{i-n}(x_{i})=0.\]
  (Notice that for each $j\in[\ell]$ $I_{j}\cap[n]\neq\emptyset$ and
  $I_{j}\setminus[n]\neq\emptyset$ by non-degeneracy and the fact that $a\neq
  0$.)\\
  \indent By \ref{axiom-eq} for all
  $j\in[\ell]$, there is a finite set $E_{j}\subset\IZ^{|I_{j}|}$ such that for all
  $i\in\IN$
  \[\bigvee_{\bar{m}\in E_{j}}\left(\bigwedge_{k\in I_{j}\cap[n]}b_{0k_{0}}=S^{m_{k}}(b_{0k})
    \wedge\bigwedge_{k\in I_{j}\setminus[n]}b_{0k_{0}}=S^{m_{k}}(b_{ik})\right),\]
  where $k_{0}=\min I_{j}$.\\
  \indent But this is a contradiction since the set defined by the formula
  \[\bigvee_{\bar{m}\in E_{j}}\left(\bigwedge_{k\in I_{j}\cap[n]}b_{0k_{0}}=S^{m_{k}}(b_{0k})
      \wedge\bigwedge_{k\in
        I_{j}\setminus[n]}b_{0k_{0}}=S^{m_{k}}(x_{k})\right)\]
  is finite for all $j\in[\ell]$.
\end{proof}
As a corollary, we obtain a uniform bound on the number of non-degenerate
solutions.
\begin{corr}
  Let $\bar{Q}\in (\IZ[X]\setminus\triv)^{n}$ and $\str{M}\models T_{R}$.
  Then there exist
  $k\in\IN$ such that
  \[\str{M}\models\forall y\left(y\neq 0\rightarrow\exists \bar{x}_{1},\ldots,\bar{x}_{k}\in R\,\forall
    \bar{z}\in R
    \left(N_{\bar{Q}}^{\nd}(\bar{z},y)\rightarrow\bigvee_{j=1}^{k}
      \bigwedge_{i=1}^{n}z_{i}=x_{ji}\right)\right).\]

\end{corr}
\begin{proof}
   Let $\op_{i}=\op_{Q_{i}}$ for all $i\in[n]$. If the corollary is false, we
   can find  a
  sequence $(a_{i}\mid i\in\IN^{>0})$ in $M\setminus\{0\}$ such that
  $\op_{1}(x_{1})+\cdots+\op_{n}(x_{n})=a_{i}$ has at least $i+1$
  non-degenerate solutions in $R(M)^{n}$. By Proposition
  \ref{prop-sol-inhom}, we may assume that $a_{i}\neq a_{i'}$ whenever $i\neq
  i'$. For all $i\in\IN$, let
  $(\bar{b}_{ij}\mid j\in[i+1])$ be a sequence of
  pairwise distinct non-degenerate solutions of
  $\op_{1}(x_{1})+\cdots+\op_{n}(x_{n})=a_{i}$. Let $U$ be a non principal
  ultrafilter over $\IN$. Let
  $\bar{b}^{U}_{j}=[(\bar{b}_{ij}\mid i\in\IN^{>0})]_{U}$. We have that
  $\bar{b}_{j}^{U}\neq\bar{b}_{j'}^{U}$ whenever $j\neq j'$. So $\str{M}^{U}$
  has infinitely many non-degenerate solutions of
  $\op_{1}(x_{1})+\cdots+\op_{n}(x_{n})=[(a_{i})]_{U}$, a contradiction with
  Proposition \ref{prop-sol-inhom} since $\str{M}^{U}\models T_{R}$.
\end{proof}
\begin{proposition}\label{prop-terms-substructures}
  Let $\str{M},\str{M}_{0}\models T_{R}$ such that $\str{M}_{0}\subset
  \str{M}$. Let $\bar{\op}$ be a tuple of $n$ non-trivial terms,
  $b_{1},\ldots,b_{n}\in R(M)\backslash R(M_{0})$ in different orbits and
  $a\in M_{0}$, $a\neq 0$. Then
  \[\sum_{i=1}^{n}\op_{i}(b_{i})+a\notin R(M).\]
\end{proposition}
\begin{proof}
  Suppose, towards a contradiction, that
  $\sum_{i=1}^{n}\op_{i}(b_{i})+a=b_{n+1}\in R(M)$. By \ref{axiom-image}, there
  exists
  $b_{n+2},\ldots,b_{2n+2}\in R(M_{0})$ such that
  $\sum_{i=1}^{n}\op_{i}(b_{i})-b_{n+1}=\sum_{i=1}^{n}\op_{i}(b_{n+1+i})-b_{2n+2}=-a$.
  By Lemma \ref{lemma-sol-eq-orbits} applied to \[\sum_{i=1}^{n}\op_{i}(x_{i})-x_{n+1}-\sum_{i=1}^{n}\op_{i}(x_{n+1+i})+x_{2n+2}=0,\]
  and $b_{1},\ldots,b_{n}$, for all $i\in[n]$, there exists
  $J_{i}\subset\{n+1,\ldots,2n+2\}$ such that $b_{i},(b_{j})_{j\in J_{i}}$ is a
  non-degenerate solution to the corresponding equation. We furthermore have
  that $J_{i}\neq \emptyset$ for all $i\in[n]$ and $J_{i}\cap J_{i'}=\emptyset$
  for all $i\neq i'\in[n]$. Furthermore, since $b_{k}\in M_{0}$ for all $k>n+1$,
  $k\notin J_{i}$ for all $i\in[n]$. So $n+1\in J_{i}$ for all $i\in[n]$.
  This implies that $n=1$. So, $\op_{1}(b_{1})-b_{2}=0$. But this contradicts
  the fact that $a\neq 0$.
\end{proof}
In the following lemma, we express the fact that the set of solutions of an
equation can be decomposed as a union of non-degenerate sets of solutions of
sub-equations, as we did in Section~\ref{section-regseq} before the statement of
Proposition \ref{proposition-sol-eq}. As the proof is a straightforward
verification, we leave the details to the reader.
\begin{lemma}\label{lemma-dec-eq}
  Let $\bar{Q}\in \IZ[X]^{n}$. Then
  \pushQED{\qed}
  \[
    T_{R}\models\forall y\forall\bar{x}\in
    R\left(\sum_{i=1}^{n}\op_{i}(x_{i})=y\leftrightarrow
      \bigvee_{\bar{I}\in\frak{P}([n])}\left(
      N^{\nd}_{\bar{Q}_{I_{0}}}(\bar{x}_{I_{0}},y)\wedge
      \bigwedge_{j=1}^{|\bar{I}|-1}N^{\nd}_{\bar{Q}_{I_{j}}}
      (\bar{x}_{I_{j}},0)\right)\right).
    \qedhere
  \]
  \popQED
\end{lemma}
We apply this to understand sets defined by $\Im$ predicates. Let
$\str{M}\models T$. Let $[Q]\in\IZ[X]^{n\times m}$,
$[Q']\in\IZ[X]^{n\times m'}$ and $D$ a set of divisibility conditions
of size $m$. Let
$\bar{a}\in M^{n}$.  We want to understand the set $N$ defined
by
\begin{equation}\label{eq:im}
  \Im_{[Q],D}\left(a_{1}+\sum_{i=1}^{m'}\op_{Q'_{1i}}(x_{i}),
    \ldots,a_{n}+\sum_{i=1}^{m'}\op_{Q'_{ni}}(x_{i})\right).
\end{equation}
This will be needed to show that $T_{R}$ is 1-e.c. We want to show that the formula (\ref{eq:im}) expresses two things for a tuple $\bar{b}\in R(M)^{m'}$:
\begin{enumerate}
\item there exists $J'\subset[m']$ such that $\bar{b}_{J'}$ belongs to a finite set depending only on $[Q],[Q'],D$ and $\bar{a}$;
\item for all $J\subset[m']$, $\bar{b}_{J}$ satisfies a finite number of recurrence relations and congruence relations again depending only on $[Q],[Q'],D$ and $\bar{a}$. Also, when $J_{1},J_{2}\subset[m]$ have a non-empty intersection, the above conditions on $\bar{b}_{J_{1}}$ and $\bar{b}_{J_{2}}$ must be consistent.
\end{enumerate}
To do so, for all
$J_{01},\ldots,J_{0n}\subset[m]$ and $J_{11},\ldots,J_{1n}\subset[m']$, we let
$N^{\nd}_{[Q],[Q'],\bar{J}_{0},\bar{J_{1}}}(\bar{a})$ be the set defined by the formula
\[
\exists\bar{z}\in
R\,\phi_{D}(\bar{z})\wedge\bigwedge_{i\in[n]}\left(\sum_{j\in
    J_{0i}}\op_{Q_{ij}}(z_{j})=a_{i}+\sum_{j\in J_{1i}}\op_{Q'_{ij}}(x_{i})\wedge
  \bar{z}_{J_{0i}}\bar{x}_{J_{1i}}\text{ is non-degenerate}\right).
\]
Notice that by Proposition \ref{prop-sol-inhom}, the set
$N^{\nd}_{[Q],[Q'],\bar{J}_{0},\bar{J_{1}}}(\bar{a})$ is finite if $a_{i}\neq 0$
for some $i\in[n]$.\\
\indent Recall that by
\ref{axiom-image}, the formula (\ref{eq:im}) is satisfied by some $\bar{b}\in R(M)^{m'}$ if and
only if there is $\bar{z}\in R(M)^{m}$ such that the following system of equations
and congruence relations is satisfied:
\[
  \begin{cases}
    \op_{Q_{11}}(z_{1})+\cdots +\op_{Q_{1m}}(z_{m})
    =a_{1}+\op_{Q'_{11}}(b_{1})+\cdots +\op_{Q_{1m'}}(b_{m'})\\
    \vdots\\
    \op_{Q_{n1}}(z_{1})+\cdots +\op_{Q_{nm}}(z_{m})
    =a_{n}+\op_{Q'_{n1}}(b_{1})+\cdots +\op_{Q_{nm'}}(b_{m'})\\
    D_{\ell_{1}}(\op_{P_{1}}(z_{1})+k_{1}),\ldots,
    D_{\ell_{m}}(\op_{P_{m}}(z_{m})+k_{m}).
  \end{cases}
\]
\indent Now, for each $i\in[n]$, choose, according to Lemma \ref{lemma-dec-eq}, $\bar{J}_{i}=(J_{i0},\ldots,J_{i\ell_{1}})\in\frak{P}([m])$ and $\bar{J'}_{i}=(J'_{i0},\ldots,J'_{i\ell_{2}})\in\frak{P}([m'])$ such that for all $i\in[n]$, the following equalities hold in a non-degenerate way
\[
  \sum_{j\in J_{i0}}\op_{Q_{ij}}(z_{j})=a_{i}+\sum_{j\in J'_{i0}}\op_{Q_{ij}}(b_{j}),
  \]
\[
  \sum_{j\in J_{is_{1}}}\op_{Q_{ij}}(z_{j})=\sum_{j\in J'_{is_{2}}}\op_{Q_{ij}}(b_{j}),
\]
for all $(s_{1},s_{2})$ in some fixed $K_{i}\subset [m]\times[m']$ and for the $s_{1}$ and $s_{2}$ that do not appear in $K_{i}$
\[
  0=\sum_{j\in J'_{is_{2}}}\op_{Q_{ij}}(b_{j})\text{ and }
  \sum_{j\in J_{is_{1}}}\op_{Q_{ij}}(z_{j})=0.
\]
This decomposition of each equation in the system shows that $\bar{b}_{J'}$ is in $N^{\nd}_{[Q],[Q'],\bar{J}_{0},\bar{J_{1}}}(\bar{a})$, where $J'=\cup_{i\in[n]}J'_{i0}$. For the homogeneous equations above, we may apply \ref{axiom-eq} to obtain the desired relations between $\bar{z}_{J_{is_{1}}}$ and $\bar{b}_{J'_{is_{2}}}$ whenever $(s_{1},s_{2})\in K_{i}$.\\
\indent To summarize, we state the
following corollary, which is an explicit statement of the above discussion.
\begin{corr}\label{corollary-im}
  Let $[Q]\in\IZ[X]^{n\times m}$ and $D$ a set of divisibility of size $m$. Let
  $[Q']\in\IZ[X]^{n\times m'}$. Let $\str{M}\models T_{R}$ and $\bar{a}\in
  M^{n}$. Then for all $\bar{b}\in R(M)^{m'}$
  \[
    \str{M}\models
      \Im_{[Q],D}\left(a_{1}+\sum_{i=1}^{m'}\op_{Q'_{1i}}(b_{i}),
        \ldots,a_{n}+\sum_{i=1}^{m'}\op_{Q'_{ni}}(b_{i})\right)
  \]
  if and only if for all $i\in[n]$ there are $(J_{i0},\ldots,J_{is})\in\frak{P}([m])$
  and $(J'_{i0},\ldots,J'_{is'})\in\frak{P}([m'])$  and
  $K_{i}\subset[s]\times[s']$ such that for all $s_{1}\in[n]$ there is at most
  one $s_{2}\in[n]$ such that $(s_{1},s_{2})\in K_{i}$ and
  \begin{align}
   \str{M}\models&\,\bar{b}_{J'}\in N^{\nd}_{[Q],[Q'],\bar{J}_{0},\bar{J'_{0}}}(\bar{a})\label{eqim:0}\\
    &\wedge
            \bigwedge_{i\in[n]}\bigwedge_{(s_{1},s_{2})\in K_{i}}\left(
            \sum_{j\in J_{is_{1}}}\op_{Q_{ij}}(S^{k_{ij}}(b_{j^{*}}))=\sum_{j\in
            J'_{is_{2}}}\op_{Q'}(S^{k'_{ij}}(b_{j^{*}}))\right.\label{eqim:1}\\
    &\left.\wedge \bigwedge_{j\in J'_{is_{2}}}b_{j}=S^{k'_{ij}}(b_{j^{*}})
            \bigwedge_{j\in J_{is_{1}}}D_{\ell_{j}}(\op_{P_{j}}(S^{k_{ij}}(b_{j^{*}}))+k'_{j})
      \right)\\
    &\wedge\bigwedge_{i\in[n]}\bigwedge_{(s_{1},s_{2})\notin K_{i}}\left(
            0=\sum_{j\in J'_{is_{2}}}\op_{Q'_{ij}}(S^{k_{ij}}(b_{j^{*}}))\wedge \bigwedge_{j\in J'_{is_{2}}}b_{j}=S^{k_{ij}}(b_{j^{*}})
      \right)\label{eqim:2}\\
    &\wedge\bigwedge_{i\in[n]}\bigwedge_{(s_{1},i',s'_{1},s'_{2})\in K'_{i}}\left(
      \sum_{j\in J_{is_{1}}}\op_{Q_{ij}}(S^{k_{ij}}(b_{j^{*}}))=0
      \wedge
      \bigwedge_{j\in J_{is_{1}}}
      D_{\ell_{j}}(\op_{P_{j}}(S^{k_{ij}}(b_{j^{*}}))+k'_{j})\right)\label{eqim:3}\\
    &\wedge\Im_{[\tilde{Q}],D}(0),
  \end{align}
  where $J'=\cup_{i\in[n]}J'_{i0}$, for all $i\in[n]$,
 $j^{*}=\min J'_{is_{2}}$ and
  \begin{enumerate}
    \item if $(s_{1},s_{2})\in K_{i}$, $\bar{k}_{i}$ and $\bar{k}'_{i}$ are given by \ref{axiom-eq}
      applied to the operators in (\ref{eqim:1});
    \item if $(s_{1},s_{2})\notin K_{i}$, $\bar{k}_{i}$ is given by
      \ref{axiom-eq} applied to the operators in (\ref{eqim:2}),
    \item $(s_{1},i',s'_{1},s'_{2})\in K'_{i}$ if and only if
      $(s_{1},s_{2})\notin K_{i}$ for all $s_{2}\in[n]$, $J_{is_{1}}\cap
      J_{i's'_{1}}\neq\emptyset$, $(s'_{1},s'_{2})\in K_{i}$ and $j^{*}=\min
      J_{i's'_{2}}$. In this case, $\bar{k}_{i}$ is given by
      \ref{axiom-eq} applied to the operators in (\ref{eqim:3}),
    \end{enumerate}
    and $[\tilde{Q}]$ is the matrix defined by $\tilde{Q}_{ij}=Q_{ij}$ if
    $Q_{ij}$ does not appear in (\ref{eqim:0})-(\ref{eqim:3}) and
    $\tilde{Q}_{ij}=0$ otherwise.\qed
\end{corr}
\subsubsection{$T_{R}$ has algebraically prime models}\label{section-alg-prime}
\newcommand{\Div}{\textrm{div}} Let $\str{M}\models T_{R}$ and
$\str{A}\subset\str{M}$. For $X\subset M$, we let $\Div(X)$ be the
\emph{divisible closure of $X$ in $\str{M}$}, that is the substructure generated
by $\{d\mid nd\in X\textrm{ for some }n\in\IN^{>0}\}$. We will show later on
that when $X$ is the domain of an $\L$-substructure, $\Div(X)$ is the divisible
closure of $X$ in the group theoretic sense. The construction of
the algebraically prime model over $\str{A}$, denoted $\overline{\str{A}}$, is
done as follows. Let $\bar{\op}$ be a $n$-tuple of non-trivial terms. Call an
$n$-tuple $\bar{b}\in R(M)$ \emph{$\bar{\op}$-good} if
\begin{enumerate}
\item $b_{i}\notin A$ for all $i\in[n]$;
\item $\op_{1}(b_{1})+\cdots+\op_{n}(b_{n})\in\str{A}$;
\item $b_{i}\notin\Orb(b_{j})$ whenever $j\neq i$.
\end{enumerate}
Let $\tilde{\str{A}}$ be the substructure generated by $\str{A}$ and
$\bar{\op}$-good tuples of elements of $R(M)$, for all tuples $\bar{\op}$
of non-trivial terms. This structure will satisfy all axioms of $T_{R}$
except the definition of the symbols $D_{n}$. So our algebraically prime model
over $\str{A}$ will be $\overline{\str{A}}=\Div(\tilde{\str{A}})$.
\begin{lemma}\label{lemma-alg-prime}
  $\overline{\str{A}}$ is a model of $T_{R}$.
\end{lemma}
\begin{proof}
  We begin with a description of elements in $\tilde{\str{A}}$. Assume
  $\tilde{\str{A}}=\langle A, (b_{\lambda})_{\lambda<\kappa}\rangle$, where
  $b_{\lambda}\notin\Orb(b_{\lambda'})$ for all $\lambda\neq \lambda'$ and each
  $b_{\lambda}$ appears in a good tuple. We want to show that any
  $d\in\tilde{\str{A}}$ can be put in the form
  $a+\sum_{i=1}^{n}\op_{i}(b_{\lambda_{i}})$, where
  $\lambda_{i}\neq\lambda_{j}$ for all $i\neq j\in[n]$ and $a\in A$.  Let
  $t(\bar{x},y)$ be the term $y+\sum_{i=1}^{n}\op_{i}(x_{i})$. We show that for
  all $a\in A$ and $b_{\lambda_{1}},\ldots,b_{\lambda_{n}}$ in different orbits,
  either $S(t(\bar{b},a))=t(\bar{b},a)$ or $t(\bar{b},a)=S^{k}(b_{\lambda})$ for
  some $\lambda<\kappa$ and $k\in\IZ$. Assume $S(t(\bar{b},a))\neq
  t(\bar{b},a)$. This implies that $b=t(\bar{b},a)\in R(M)$. Then, since
  $a=b-\sum_{i=1}^{n}\op_{i}(b_{\lambda_{i}})$, either $b$ is in the orbit of
  $b_{\lambda_{i}}$ for some $i\in[n]$ or $(b,\bar{b})$ is an $(x,-\bop)$-good
  tuple. This shows that $b=S^{k}(b_{\lambda})$ for some $\lambda<\kappa$ and
  $k\in\IZ$. Thus every element $\tilde{\str{A}}$ is of the form
  $a+\sum_{i=1}^{n}\op_{i}(b_{\lambda_{i}})$.\\
  \indent We now do the same job for elements in $\overline{\str{A}}$.
  \begin{claim}\label{claim-elt}
    Let $d\in\overline{\str{A}}$. Then there
    exist $a\in\tilde{\str{A}}$ and $n\in\IN^{>0}$ such that $nd=a$.
  \end{claim}
  \begin{proof}
    Let $X$ be the set $\{d\mid nd\in \tilde{A}\textrm{ for some
    }n\in\IN^{>0}\}$. We first notice that for all $\bar{d}\in X^{k}$ and
    $\bar{m}\in \IZ^{k}$, there is $n\in\IN$ such that
    $n(m_{1}d_{1}+\cdots+m_{k}d_{k})\in\tilde{A}$ (just take $n$ to be the
    product of the witnesses of the fact that $\bar{d}\in X^{k}$). So to
    conclude, it is enough to show that for all terms $t(\bar{x})$,
    $|x|=k$, and
    all $\bar{d}\in X^{k}$' either $t(\bar{d})\in \tilde{A}$ or there is
    $\bar{m}\in\IZ^{k}$ such that
    $t(\bar{d})=m_{1}d_{1}+\cdots+m_{k}d_{k}$. This is done by induction on the
    complexity of terms (the complexity being here the number of occurrences of
    the symbols $S$ and $S^{-1}$).\\
    \indent Let $t(\bar{x})$ be a term, $|x|=k$, and
    $\bar{d}\in X^{k}$ and
    assume that either $t(\bar{d})\in \tilde{A}$ or there is
    $\bar{m}\in\IZ^{k}$ such that $t(\bar{d})=m_{1}d_{1}+\cdots+m_{k}d_{k}$. We
    may assume that $t(\bar{d})=m_{1}d_{1}+\cdots+m_{k}d_{k}\notin\tilde{A}$, since
    $\tilde{\str{A}}$ is closed under $S$ and $S^{-1}$. We claim that
    $S(t(\bar{d}))=t(\bar{d})$. Indeed, since $\bar{d}\in X^{k}$, there exists
    $n\in\IN$ such that $nt(\bar{d})\in\tilde{A}$. Thus, if
    $S(t(\bar{d}))=t(\bar{d})$ is not true, then $t(\bar{d})\in R(M)$ and this
    imply that $t(\bar{d})$ is $(nx)$-good, in contradiction with our
    assumption that $t(\bar{d})\notin\tilde{A}$.
  \end{proof}
  Let us finally show that $\overline{\str{A}}\models T_{R}$. The only axiom
  that requires details is \ref{axiom-image} -- the defining axioms for the
  divisibility predicates are true since we took the divisible closure of
  $\tilde{\str{A}}$ and the others are universal and thus
  true in any substructure. So assume that
  $\overline{\str{A}}\models\Im_{[Q],D}(d_{1},\ldots,d_{n})$, where
  $[Q]\in(\IZ[X])^{n\times m}$ and $D$ is a set of divisibility conditions
  of size $m$. By Claim
  \ref{claim-elt}, we may assume that $d_{1},\ldots,d_{n}\in\tilde{\str{A}}$:
  for all $i\in[n]$, $d_{i}=a_{i}+\sum_{j=1}^{k}\op'_{ij}(d_{\lambda_{ij}})$. So
  we may also assume that $d\in A$. Since $\str{M}\models T_{R}$ and we can
  find $b_{1},\ldots,b_{m}\in R(M)$ such that
  \[\str{M}\models \bigwedge_{i\in [n]}\sum_{j\in[m]}\op_{Q_{ij}}(b_{j})=d_{i}
    \wedge\bigwedge_{i\in[m]}D_{\ell_{i}}(\op_{Q_{i}}(b_{i})+k_{i}).\]
  We may assume that for all $j\in[m]$ there exists $i\in[n]$ such that
  $\op_{Q_{ij}}$ is non-trivial
  (if for some $j\in[m]$ the terms $\op_{Q_{ij}}$ are trivial for all $i\in[n]$, we
  may replace, by \ref{axiom-r0} and \ref{axiom-operator}, $b_{j}$ by any
  $b'_{i}\in R(\IZ)$ such
  that $\str{Z}_{R}\models D_{\ell_{j}}(\op_{j}(b'_{j})+k_{j})$, which is
  possible since $\str{Z}_{R}\subset\str{A}$ and $\str{Z}_{R}\models
  T_{R}$). Now by construction of $\overline{\str{A}}$, for all $i\in[m]$,
  $b_{i}$ is in the orbit of some $b_{\lambda_{i}}$. This implies that
  $\bar{b}\in \overline{A}^{m}$, as desired.
\end{proof}
Let us show that any embedding $f:\str{A}\to\str{N}$ extends to an embedding
$\bar{f}:\overline{\str{A}}\to\str{N}$.
\begin{lemma}
  Let $f:\str{A}\to\str{N}$ be an $\L$-embedding. Then $f$ extends to an
  $\L$-embedding $\bar{f}:\overline{\str{A}}\to\str{N}$.
\end{lemma}
\begin{proof}
  Let $\L_{0}$ be the language $\{+,-,0,1,R\}\cup\L_{S}$. We first extend $f$ to
  an $\L_{0}$-embedding $\tilde{f}:\tilde{\str{A}}\to\str{N}$. Let $q$ be the
  partial type
  \begin{align*}
    &\{\op_{1}(x_{\lambda_{1}})+\cdots+\op_{n}(x_{\lambda_{n}})=
      f(a)\mid \str{M}\models
      \op_{1}(b_{\lambda_{1}})+\cdots+\op_{n}(b_{\lambda_{n}})=a,
      a\in A\}\\
    \cup&\{x_{\lambda}\neq f(a)\mid \lambda<\kappa,a\in A\}\\
    \cup& \{S^{k}(x_{\lambda_{1}})\neq x_{\lambda_{2}}\mid \lambda_{1}\neq\lambda_{2},z\in\IZ\}\\
    \cup& \{D_{\ell}(\op(x_{\lambda})+k)\mid \str{M}\models
          D_{\ell}(\op(b_{\lambda}+k))\}.
  \end{align*}
  We claim that $q$ is finitely consistent in $\str{N}$. Let $\Delta$ be a
  finite part of $q$. We may assume that the conjunction of the formulas in $\Delta$
  is of the form
  \begin{align*}
    &\bigwedge_{i\in
      I_{1}}\op_{i1}(x_{\lambda_{1}})+\cdots+\op_{in}(x_{\lambda_{n}})=f(a_{i})
    \wedge \bigwedge_{i\in[n]}D_{\ell_{i}}(\op_{i}(x_{\lambda_{i}})+k_{i})\\
    \wedge &\bigwedge_{i\in I_{2};j\in[n]}x_{\lambda_{j}}\neq f(a_{ij})
    \wedge \bigwedge_{i,j\in[n];k\in I_{3}}S^{k}(x_{\lambda_{i}})\neq x_{\lambda_{j}}.
  \end{align*}
  By \ref{axiom-image}, there exists $\bar{b}'\in R(N)^{n}$ such that
 \begin{align*}
    &\bigwedge_{i\in
      I_{1}}\op_{i1}(b'_{{1}})+\cdots+\op_{in}(b'_{{n}})=f(a_{i})
    \wedge \bigwedge_{i\in[n]}D_{\ell_{i}}(\op_{i}(b'_{{i}})+k_{i}).
 \end{align*}
 Assume towards a contradiction that $\bar{b}'$ is not a realization of
 $\Delta$. Then we have that, for some $i_{1}\in I_{2}$,
 $j_{1},i_{2},j_{2}\in[n]$ and $k\in I_{3}$,
 \[b'_{j_{1}}=f(a_{i_{1}j_{1}})\vee S^{k}(b'_{i_{2}})=b'_{j_{2}}.\]
 So, again using \ref{axiom-image}, we can find $\bar{b}''\in R(M)^{n}$ such
 that
  \begin{align*}
    &\bigwedge_{i\in
      I_{1}}\op_{i1}(b''_{{1}})+\cdots+\op_{in}(b''_{{n}})=a_{i}
    \wedge \bigwedge_{i\in[n]}D_{\ell_{i}}(\op_{i}(b''_{{i}})+k_{i})
    \wedge(b''_{j_{1}}=a_{i_{1}j_{1}}\vee S^{k}(b''_{i_{2}})=b''_{j_{2}}),
  \end{align*}
  in contradiction with the fact that $b_{\lambda_{1}},\ldots,b_{\lambda_{n}}$
  is a good tuple. Hence $q$ is finitely consistent in $\str{N}$ and so
  realized in an elementary extension $\str{N}^{*}$ of
  $\str{N}$ by some $(b'_{\lambda})_{\lambda<\kappa}$. Let us show that
  $(b'_{\lambda})_{\lambda<\kappa}$ is in $\str{N}$. Let $\lambda<\kappa$. By
  definition, $b_{\lambda}$ appears in a $\bar{\op}$-good tuple: there exist
  $b_{\lambda_2},\ldots,b_{\lambda_{n}}\in R(M)\backslash A$ and $a\in A$ such
  that
  $\op_{1}(b_{\lambda})+\op_{2}(b_{\lambda_{2}})+\cdots+\op_{n}(b_{\lambda_{n}})=a$.
  The same holds for $b'_{\lambda},b'_{\lambda_{2}},\ldots,b'_{\lambda_{n}}$ and
  $f(a)$. Furthermore, we have that
  $\str{N}\models\Im_{\bar{Q}}(f(a))$, where $\op_{i}=\op_{Q_{i}}$. Since
  $\str{N}\models T_{R}$, there
  are $d_{1},\ldots,d_{n}\in R(N)$ such that
  \[\sum_{i=1}^{n}\op_{i}(d_{i})=f(a).\]
  Hence, by Lemma \ref{lemma-sol-eq-orbits}, $b'_{\lambda}$ is in the orbit of
  $d_{i}$ for some $i\in[n]$: this shows that $b'_{\lambda}\in N$.\\
  \indent Since for all $\lambda_{1}\neq\lambda_{2}$ and all $z\in\IZ$, the
  formula $S^{k}(x_{\lambda_{1}})\neq x_{\lambda_{2}}$ is in $q$, we have that
  for  all $\lambda_{1}\neq\lambda_{2}$,
  $b'_{\lambda_{1}}\notin\Orb(b'_{\lambda_{2}})$. Likewise, we have
  that $b'_{\lambda}\notin f(A)$ for all $\lambda<\kappa$. This shows
  that $(b'_{\lambda})_{\lambda<\kappa}$ realizes the quantifier-free type of
  $(b_{\lambda})_{\lambda<\kappa}$ over $A$ in $\L_{0}$. Hence the map
  $\tilde{f}$ defined on $\tilde{\str{A}}$ by
  $a+\sum_{i=1}^{n}\op_{i}(b_{\lambda_{i}})\mapsto
  f(a)+\sum_{i=1}^{n}\op_{i}(b'_{\lambda_{i}})$ is an $\L_{0}$-embedding.\\
  \indent Now we extend $\tilde{f}$ to an $\L$-embedding
  $\bar{f}:\overline{\str{A}}\to \str{N}$.  Recall that for all
  $d\in\overline{\str{A}}\backslash\tilde{\str{A}}$, there exist $a\in\str{A}$,
  $\bar{\op}$ a tuple of non-trivial terms,
  $b_{\lambda_{1}},\ldots,b_{\lambda_{n}}$ and $n\in\IN^{>0}$ such that
  $nd=a+\sum_{i=1}^{n}\op_{i}(b_{\lambda_{i}})$. By construction $\tilde{f}(nd)$
  is divisible by $n$: by \ref{axiom-group} there exists a unique $d^{*}$ such
  that $\tilde{f}(nd)=nd^{*}$ (uniqueness follows from the fact that models of
  $T_{1}$ are torsionless). We extend $\tilde{f}$ by the rule
  $\bar{f}(d)=d^{*}$. So $\bar{f}$ respects the divisibility predicates. And
  since the $\Im$ predicates are definable by
  $\L_{0}\cup\{D_{n}\mid n\in\IN^{>1}\}$-formulas, we get that $\bar{f}$ is indeed
  an $\L$-embedding
\end{proof}
\subsubsection{$T_{R}$ is 1-e.c.}\label{section-1ec}
We will need an analogue of Proposition \ref{proposition-image-op}.
\begin{proposition}\label{proposition-op-nstd}
  Let $Q_{1},\ldots,Q_{k}\in\IZ[X]$. Then for all $\str{M}\models T_{R}$,
  $\{z\in\IN\mid \str{M}\models \exists\bar{x}\in R\;
  \op_{Q_{1}}(x_{1})+\cdots+\op_{Q_{k}}(x_{k})=z\}$ is not piecewise syndetic.
\end{proposition}
\begin{proof}
  This is an immediate consequence of Proposition \ref{proposition-image-op} and
  the following observation:
  \begin{align*}
    &\{z\in\IN\mid \str{M}\models \exists\bar{x}\in R\;
      \op_{Q_{1}}(x_{1})+\cdots+\op_{Q_{k}}(x_{k})=z\}\\
    =&\{z\in\IN\mid \str{Z}_{R,\L}\models \exists\bar{x}\in R\;
  \op_{Q_{1}}(x_{1})+\cdots+\op_{Q_{k}}(x_{k})=z\}.
  \end{align*}
\end{proof}
\begin{proposition}\label{prop-mc-1}
  Let $\str{M},\str{M}_{0}\models T_{R}$ such that $\str{M}_{0}\subset
  \str{M}$. Assume that $R(M_{0})=R(M)$. Then $\str{M}_{0}$ is 1-e.c. in
  $\str{M}$.
\end{proposition}
\begin{proof}
  Let $\phi(x,\bar{y})$ be a quantifier-free formula such that
  $\str{M}\models \phi(b,\bar{a})$ for some $b\in M\backslash M_{0}$ and
  $\bar{a}\in M_{0}$. We will show that there exists $b_{0}\in M_{0}$ such that
  $\str{M}_{0}\models\phi(b_{0},\bar{a})$. Let us simplify $\phi$.\\
  \indent First we show that for all $\L$-terms $t(x,\bar{y})$, $\bar{y}$ of
  size $n$, for all $b\in M\setminus M_{0}$ and all $\bar{a}\in M_{0}^{n}$,
  there are $n\in\IZ$ and $a\in M_{0}$ such that $t(b,\bar{a})=nb+a$. It is
  enough to show that for all $n\in\IZ\setminus\{0\}$, $b\in M\setminus M_{0}$
  and $a\in M_{0}$, $S(nb+a)=nb+a$. But this is the case since
  $nb+a\notin M_{0}$ and $R(M)=R(M_{0})\subset M_{0}$.  In particular for all
  $b\in M\setminus M_{0}$, $n\in\IZ$ and $a\in M_{0}$, $\str{M}\models R(nb+a)$
  if and only if $n=0$ and
  $\str{M}\models R(a)$.\\
  \indent Now we look at the atomic formulas satisfied by elements in
  $M\setminus M_{0}$ with parameters in $M_{0}$. Let $b\in M\setminus M_{0}$,
  $n_{1},\ldots,n_{k}\in\IZ$, $a_{1},\ldots,a_{k}\in M_{0}$,
  $[Q]\in(\IZ[X])^{k\times m}$ and $D$ be a set of divisibility conditions of
  size $m$. Since $R(M)=R(M_{0})$, we have
  $\str{M}\models \Im_{[Q],D}(n_{1}b+a_{1},\ldots,n_{k}b+a_{k})$ if and only
  if $n_{1}=\cdots=n_{k}=0$ and $\str{M}\models\Im_{[Q],D}(\bar{a})$.
  Likewise, for all $n\in\IZ$ and $a\in M_{0}$, we have $\str{M}\models
  nb+a=0$ if and only if $n=0$ and $\str{M}\models a=0$.\\
  \indent Thus, after writing $\phi(x,\bar{y})$ in its equivalent disjunctive
  normal form, we may select a conjunctive clause satisfied by $(b,\bar{a})$ and
  then assume that $\phi(x,\bar{a})$ is of the form
  \begin{align*}
    &\bigwedge_{i\in I_{1}} n_{i}x+a'_{i}\neq 0\\
    \wedge &\bigwedge_{i\in I_{2}} \neg\Im_{[Q]_{i},D_{i}}(n_{i1}x+a'_{i1},
             \ldots,n_{im_{i}}x+a'_{im_{i}})\\
    \wedge &\bigwedge_{i\in I_{3}} D_{\ell_{i}}(n_{i}x+k_{i}),
  \end{align*}
  where for all $i\in I_{1}$, $n_{i}\in \IZ\setminus\{0\}$ and
  $a'_{i}=t_{i}(\bar{a})$ for some $\L$-term $t_{i}(\bar{y})$,
  for all $i\in I_{3}$, $m_{i}\in\IN^{>0}$,
  $\bar{n}_{i}\in(\IZ\setminus\{0\})^{m_{i}}$ and $a'_{ij}=t_{ij}(\bar{a})$ for
  some $\L$-term $t_{ij}(\bar{y})$
  and for all $i\in I_{3}$, $n_{i}\in\IZ\setminus\{0\}$, $\ell_{i}\in\IN^{>1}$
  and $0\le k_{i}<\ell_{i}$.\\
  \indent Let us finally show that $\phi(M_{0},\bar{a})$ is not empty.  By model
  completeness of $\Th(\IZ,+,0,1,D_{n}\mid 1<n\in\IN)$, there exists
  $b_{0}\in M_{0}$ such that
  \[\str{M}_{0}\models\bigwedge_{i\in
      I_{1}}n_{i}b_{0}+a'_{i}\neq 0\wedge\bigwedge_{i\in
      I_{3}}D_{\ell_{i}}(n_{i}b_{0}+a'_{i}).\]
  However, $\str{M}_{0}$ may not satisfy $\phi(b_{0},\bar{a})$. But this can be
  overcome in the following way. Let
  \[X_{1}=\left\{m\in\IN\,\left|\, \str{M}_{0}\models\bigwedge_{i\in
          I_{3}}n_{i}(b_{0}+m)+a'_{i}\neq 0\right.\right\},\]
  \[X_{2}=\left\{m\in\IN\,\left|\, \str{M}_{0}\models\bigwedge_{i\in
          I_{2}}\neg\Im_{[Q]_{i},D_{i}}(n_{i1}(b_{0}+m)+a'_{i1},\ldots)\right.\right\},\]
  and
  \[
    X_{3}=\left\{m\in\IN\,\left|\, \str{M}_{0}\models\bigwedge_{i\in
          I_{3}}D_{\ell_{i}}(n_{i}m)\right.\right\}.
  \]
  We want to show that the set $X=X_{1}\cap X_{2}\cap X_{3}$ is not empty.
  Suppose otherwise that $X=\emptyset$. This implies that $X_{3}\subset
  \IN\setminus (X_{1}\cap X_{2})$. But then, since $X_{3}$ is piecewise syndetic,
  $\IN\setminus(X_{1}\cap X_{2})$ is piecewise syndetic. Hence, by Brown's Lemma,
  $\IN\setminus X_{2}$ is piecewise syndetic, $\IN\setminus X_{1}$ being
  finite. But, by Proposition \ref{proposition-op-nstd}, this is not possible since
  $\IN\setminus X_{2}$ is in the image of a sum of terms of the form $\op_{Q}$.
\end{proof}

In order to establish that $T_{R}$ is 1-e.c., we first show that a conjunction
of $\Im$ predicates is equivalent to an $\Im$ predicate.
\begin{lemma}\label{lemma-im-conj}
  For all $[Q]_{1}\in\IZ^{n_{1}\times m_{1}}$, \ldots,
  $[Q]_{\ell}\in\IZ^{n_{\ell}\times m_{\ell}}$ and sets of divisibility conditions
  $D_{1},\ldots,D_{\ell}$, there exists $[Q]\in
  \IZ^{(n_{1}\cdots n_{\ell})\times(m_{1}\cdots m_{\ell})}$ and a set of divisibility
  conditions $D$ such that
  \[
    T_{R}\models\forall \bar{y}_{1},\ldots,\bar{y}_{\ell}\in R
    \left(
      \bigwedge_{i\in[\ell]}\Im_{[Q]_{i},D_{i}}(\bar{y}_{i})
      \leftrightarrow \Im_{[Q],D}(\bar{y}_{1},\ldots,\bar{y}_{\ell})
    \right).
  \]
\end{lemma}
\begin{proof}
  Just take $[Q]=[Q]_{1}\oplus\cdots\oplus[Q]_{\ell}$ ($\oplus$ denotes the
  direct sum) and $D=D_{1}\sqcup\cdots\sqcup D_{\ell}$.
\end{proof}

\begin{theorem}\label{thm-mc}
  The theory $T_{R}$ is 1-e.c.
\end{theorem}
\begin{proof}
  Let us show that for all $\str{M},\str{M}_{0}\models T_{R}$ such that
  $\str{M}_{0}\subset \str{M}$, then $\str{M}_{0}$ is 1-e.c. in $\str{M}$. Let
  $\str{M},\str{M}_{0}\models T_{R}$ such that $\str{M}_{0}\subset \str{M}$. Two
  cases are possible: either $R(M_{0})=R(M)$ or $R(M_{0})\subsetneq R(M)$. The
  first case has been proved in
  Proposition \ref{prop-mc-1}. So let us assume that we are in the second case.\\
  \indent By Lemma \ref{lemma-alg-prime}, we may assume that
  $\str{M}=\overline{\str{A}}$ where $\str{A}$ is the substructure of $\str{M}$
  generated by $M_{0}\cup R(M)$. Recall that by the proof of
  Lemma \ref{lemma-alg-prime}, any element $d$ of $\str{M}$ is such that
  $nd=a+\sum_{i=1}^{\ell}\op_{i}(b_{i})$, where $n\in\IN$, $a\in M_{0}$ and
  $b_{1},\ldots,b_{\ell}\in R(M)\backslash R(M_{0})$ are in different orbits. Our
  strategy is to
  establish that $\str{M}_{0}$ is 1-e.c. in $\str{M}$ from the fact that for all tuple
  $\bar{b}$ of elements of $R(M)\backslash R(M_{0})$ in different orbits, all
  $\bar{a}\in M_{0}$ and all $\phi(\bar{x},\bar{y})$,
  $\str{M}\models \phi(\bar{b},\bar{a})$ implies
  $\str{M}_{0}\models\exists\bar{x}\in R\, \phi(\bar{x},\bar{a})$.\\
  \indent Let us first look at terms evaluated at $d\in M$.
  \begin{claim}
    Let $t(x,\bar{y})$ be an $\L$-term, where $\bar{y}$ is a tuple of size
    $k$. Then for all $d\in M$ and $\bar{a}\in M_{0}^{k}$, there are $m\in\IZ$
    and $a\in M_{0}$ such that $t(d,\bar{a})=md+a$.
  \end{claim}
  \begin{proof}
    Let $d\in M$, $n\in\IN^{>0}$, $a\in M_{0}$ and $b_{1},\ldots,b_{\ell}\in
    R(M)$ in different orbits such that
    $nd=a+\sum_{i=1}^{\ell}\op_{i}(b_{i})$, with $\op_{i}=\op_{Q_{i}}$ for some
    $Q_{i}\in\IZ[X]$. To prove the claim, it is enough to
    show that $S(md+a')=md+a'$ for all $m\in\IZ\setminus\{0\}$ and $a'\in
    M_{0}$, unless $m=0$ or $d\in M_{0}$. Without loss of generality, we assume
    that $\op_{i}$ is non-trivial for all $i\in[\ell]$.
    Assume $md+a'=b\in R(M)$, $m\neq 0$. Then we have
    that $n(md+a')=ma+a'+\sum_{i=1}^{\ell}m\op_{i}(b_{i})=nb$. Thus
    \[\str{M}\models \Im_{nQ,-m\bar{Q}}(ma+na')\]
    and since $\str{M}_{0}\subset\str{M}$ and $\str{M}_{0}\models T_{R}$, there
    are $b',b'_{1},\ldots,b'_{\ell}\in R(M_{0})$ such that
    \[nb-\sum_{i=1}^{\ell}m\op_{i}(b_{i})-
      \left(nb'-\sum_{i=1}^{\ell}m\op_{i}(b'_{i})\right)=0.\]
    But by Lemma \ref{lemma-sol-eq-orbits}, this implies that
    $b,b_{1},\ldots,b_{\ell}\in M_{0}$. In particular $d\in M_{0}$.
  \end{proof}
  \indent Now, let $\phi(x,\bar{y})$ an $\L$-formula, with $\bar{y}$ of size
  $k$, such that $\str{M}\models\phi(d,\bar{a})$, for some $d\in M\setminus
  M_{0}$ and $\bar{a}\in M_{0}^{k}$. Using the previous claim, we may assume
  that $\phi(x,\bar{a})$ is of the form
  \begin{align*}
    &\bigwedge_{i\in I_{1}}m_{i}x+a'_{i}\neq 0\wedge
      \bigwedge_{i\in I_{2}}D_{\ell_{i}}(m_{i}x+s_{i})\\
    \wedge&\bigwedge_{i\in
            I_{3}}\Im_{[Q]_{i},D_{i}}(m_{i1}x+a'_{i1},\ldots,m_{ik_{i}}x+a'_{ik_{i}})\\
    \wedge&\bigwedge_{i\in
            I_{4}}\neg\Im_{[Q]_{i},D_{i}}(m_{i1}x+a'_{i1},
            \ldots,m_{ik_{i}}x+a'_{ik_{i}}),
  \end{align*}
  where, for all $i\in I_{1}\cup I_{2}$, $a'_{i}=t_{i}(\bar{a})$ for some
  $\L$-term $t_{i}(\bar{y})$, $m_{i}\in\IZ\setminus\{0\}$,
  $\ell_{i}\in\IN^{>1}$, $0\le s_{i}<\ell_{i}$ and for all $i\in I_{3}\cup
  I_{4}$, $k_{i}\in\IN^{>0}$, $a'_{ij}=t_{i}(\bar{a})$ for some $\L$-term
  $t_{ij}(\bar{y})$, $\bar{m}_{i}\in(\IZ\setminus\{0\})^{k_{i}}$,
  $[Q]_{i}\in(\IZ[X])^{k_{i}\times k'_{i}}$ and $D_{i}$ is a set of divisibility
  conditions. We may also assume that $|I_{3}|\le 1$ by Lemma \ref{lemma-im-conj}.\\
  \indent Since $d\in M\setminus M_{0}$, $nd=a+\sum_{i=1}^{\ell}\op_{i}(b_{i})$, for
  some $n\in\IN^{>0}$, $a\in M_{0}$, $b_{1},\ldots,b_{\ell}$ in different orbits
  and $\op_{1},\ldots,\op_{\ell}$ non-trivial. Since
  $D_{n_{1}n_{2}}(n_{1}x)\leftrightarrow D_{n_{2}}(x)$ and
  $\Im_{[Q],D}(\bar{y})\leftrightarrow\Im_{n[Q],D}(n\bar{y})$, instead of
  looking at $\phi(x,\bar{a})$ we may look at the formula
  $\tilde{\phi}(\bar{x},\bar{a})$ defined by
  \begin{align*}
    &\bigwedge_{i\in I_{1}}m_{i}\left(a+\sum_{i=1}^{\ell}\op_{i}(x_{i})\right)+na'_{i}\neq 0\wedge
      \bigwedge_{i\in I_{2}}D_{n\ell_{i}}\left(m_{i}\left(a+\sum_{i=1}^{\ell}\op_{i}(x_{i})\right)+ns_{i}\right)\\
    \wedge&\bigwedge_{i\in
            I_{3}}\Im_{n[Q]_{i},D_{i}}\left(m_{i1}\left(a+\sum_{i=1}^{\ell}\op_{i}(x_{i})\right)+na'_{i1},
            \ldots,m_{ik_{i}}\left(a+\sum_{i=1}^{\ell}\op_{i}(x_{i})\right)+na'_{ik_{i}}\right)\\
    \wedge&\bigwedge_{i\in
            I_{4}}\neg\Im_{n[Q]_{i},D_{i}}\left(m_{i1}\left(a+\sum_{i=1}^{\ell}\op_{i}(x_{i})\right)+na'_{i1},
            \ldots,m_{ik_{i}}\left(a+\sum_{i=1}^{\ell}\op_{i}(x_{i})\right)+na'_{ik_{i}}\right).
  \end{align*}
  Furthermore, we may replace
  \[
    \bigwedge_{i\in
      I_{2}}D_{n\ell_{i}}\left(m_{i}\left(a+\sum_{i=1}^{\ell}\op_{i}(x_{i})\right)+ns_{i}\right)
  \]
  by
  \[
    \bigwedge_{i\in
      [\ell]}D_{\ell'_{i}}\left(m'_{i}\op_{i}(x_{i})+s'_{i}\right),
  \]
  where for all $i\in [\ell]$, $\ell_{i}\in\IN^{>1}$,
  $m'_{i}\in\IZ$ and $0\le s'_{i}<\ell'_{i}$. Finally, by Lemma
  \ref{lemma-dec-eq} and Corollary
  \ref{corollary-im}, we may assume that $\tilde{\phi}(\bar{x},\bar{a})$ is
  of the form
  \begin{align*}
      &\bigwedge_{i\in[n]}\bigwedge_{j\in J}D_{\ell_{ij}}(\op_{Q_{ij}}(S^{k_{j}}(x_{i}))+k_{ij})\\
    \wedge&\bigwedge_{(i,j)\in K_{1}}\op_{Q'_{j}}(x_{i})=0
            \wedge\bigwedge_{(i,j)\in K_{2}}\op_{Q'_{j}}(x_{i})\neq 0
    \wedge\bigwedge_{i\in I}\bar{x}_{J_{i}}\notin F_{i},
  \end{align*}
  where, for all $i\in I$, $F_{i}$ is a finite set of $|J_{i}|$-tuples in
  $M_{0}$. But then, by \ref{axiom-cong} and \ref{axiom-operator}, we may find a
  realization $\bar{b}_{0}$ of $\tilde{\phi}(\bar{x},\bar{a})$ in $R(M_{0})$, as
  desired.
\end{proof}
\subsection{Superstability}
From the quantifier elimination of $T_{R}$, we deduce, by means of counting of
types, that it is superstable.
\begin{theorem}\label{theorem-sup2}
  The theory $T_{R}$ is superstable.
\end{theorem}
\begin{proof}
  Let $\str{C}$ be a monster model of $T_{R}$ and let $A\subset C$ be a small
  set of parameters. We want to show that
  $|S_{1}(A)|\le\max\{2^{\aleph_{0}},|A|\}$. Without loss of generality, we may
  assume that $A$ is the domain of a model. By quantifier elimination (see
  Theorem \ref{thm-eq}), any type $p(x)$ over $A$ is determined by the set of
  atomic formulas it contains. Let $\L_{1}=\L_{g}\cup\L_{S}$ and $\L_{2}$ be
  $\L\backslash\{D_{n}\mid n>1\}$. Let $p_{|\L_{i}}$ denote the restriction of $p$
  to $\L_{i}$, so that $p(x)=p_{|\L_{1}}(x)\cup p_{|\L_{2}}(x)$. We may assume
  that $p(x)$ does not contain a formula of the form $x=a$ for some $a\in M$. We
  consider two cases:
  \begin{enumerate}[label=(Case \arabic*)]
  \item\label{case-1} there exist $m\in\IZ\backslash\{0\}$ and $a\in A$ such
    that $R(mx+a)\in p(x)$;
  \item\label{case-2} for all $m\in\IZ\backslash\{0\}$ and all $a\in A$,
    $R(mx+a)\notin p(x)$.
  \end{enumerate}
  \begin{claim}
    Let $t(x,\bar{y})$ be an $\L$-term, with $\bar{y}$ of size $n$. Then for all
    $d\in C\setminus A$ and $\bar{a}\in A^{n}$ one of the following holds:
    \begin{enumerate}
    \item if $md+a=b\in R(C)$ for some $m\in\IZ\setminus\{0\}$ and $a\in A$,
      then there exist $Q\in\IZ[X]$, $m'\in\IZ$ and $a'\in A$ such that
      $t(d,\bar{a})=\op_{Q}(b)+m'x+a'$;
    \item if for all $m\in\IZ\setminus\{0\}$ and all $a\in A$,
      $md+a\notin R(C)$, then there exist $m'\in\IZ$ and $a'\in A$ such that
      $t(d,\bar{a})=m'd+a'$.
    \end{enumerate}
  \end{claim}
  \begin{proof}
    Let $d\in C\setminus A$.
    \begin{enumerate}
    \item Assume $md+a=b\in R(C)$ for some $m\in\IZ\setminus\{0\}$ and $a\in
      A$. It is enough to show that for all $Q\in\IZ[X]$
      $m'\in\IZ\setminus\{0\}$ and all
      $a'\in A$, if $\op_{Q}(b)+m'b+a'=b'\in R(C)\setminus R(A)$, then there exists $k\in\IZ$
      such that $\op_{Q}(b)+m'b+a'=S^{k}(b)$. Notice that $\op_{Q}(b)+m'b+a'=b'$
      is equivalent to $m\op_{Q}(b)+m'b+ma'-m'a=mb'$. Let
      $\op'(x)=m\op_{Q}(x)+m'x$, so that $\op'(b)-mb'=m'a-ma'$. Notice that
      $\op'$ is non-trivial, since $b'\in R(C)\setminus R(A)$.
      Since $\str{A}$ is a model of $T_{R}$, we can find
      $b_{0},b_{0}'\in R(A)$ such that $\op'(b_{0})-mb'_{0}=m'a-ma'$.
      Since
      $b,b'\in R(C)\setminus R(A)$, this implies by Lemma
      \ref{lemma-sol-eq-orbits} that $m'a-ma'=0$. As $\op'$ is non-trivial and
      $m\in\IZ\setminus\{0\}$,
      $(b,b')$ is a non-degenerate solution of $\op'(x)-my=0$. So, by
      \ref{axiom-eq}, there exists $k\in\IZ$ such that $b'=S^{k}(b)$, which is
      what we wanted.
    \item Assume that for all $m\in\IZ\setminus\{0\}$ and all $a\in A$,
      $md+a\notin R(C)$. In that case $S(md+a)=md+a$ for all
      $m\in\IZ\setminus\{0\}$ and all $a\in A$. This is enough to
      conclude.\qedhere
    \end{enumerate}
  \end{proof}
  \indent Using the previous claim, we may assume in the rest of the proof that
  the terms (with parameters in $A$) that appear in formulas are of the form
  $\op_{Q}(mx+a)+m'x+a'$, where $m'\in\IZ$, $a'\in A$ and
  \begin{enumerate}
  \item $m\in\IZ\setminus\{0\}$ and $a\in A$ are fixed when we are in
    \ref{case-1};
  \item $Q=0$ when we are in
    \ref{case-2}.
  \end{enumerate}
  \begin{claim}
    The number of types of the form $p_{|\L_{1}}(x)$ is at most
    $\max\{2^{\aleph_{0}},|A|\}$.
  \end{claim}
  \begin{proof}
    Indeed, any formula of the form $D_{n}(\op_{Q}(mx+a)+m'x+a')$ is equivalent
    to a
    formula of the form $D_{n}(\op_{Q}(mx+a)+m'x+k)$, where $k\in\IZ$ is such
    that
    $D_{n}(a'-k)$. In \ref{case-2}, we know that a formula of the form
    $m'x+a'=0$ is never in $p(x)$, unless $m'=0$ and $a'=0$. Let us now look at
    equations when we are in \ref{case-2}.\\
    \indent Assume that $\op_{Q}(mx+a)+m'x+a'=0\in p_{|\L_{1}}(x)$, where
    $m'\in\IZ$
    and $a'\in
    A$. Then, by axiom \ref{axiom-image}, $\Im_{mQ,m'X}(m'a-ma')$ holds in
    $\str{A}$. Thus there exists $b'\in R(A)$ such that $\op(mx+a)+m'(mx+a)=\op(b')+m'b'$. This
    implies, by axiom \ref{axiom-operator} that $\op(mx+a)+m'(mx+a)=0$. Hence
    $m'a-ma'=0$. So
    the only equations that appear in $p(x)$ are of the form $\op(mx+a)=0$.
  \end{proof}
  By the previous claim, it remains to show that the number of types of the form
  $p_{|\L_{2}}(x)$ is at most $\max\{2^{\aleph_{0}},|A|\}$. So we need to look
  at formulas of the form
  $\Im_{[Q],D}(\op_{1}(mx+a)+m_{1}x+a_{1},\ldots,\op_{k}(mx+a)+m_{k}x+a_{k})$. For
  simplicity, we only look at the case $k=1$. We may restrict ourselves to
  formulas of the form $\Im_{\bar{Q},D}(\op(mx+a)+a')$ in \ref{case-1} and
  $\Im_{\bar{Q},D}(m'x+a')$ in \ref{case-2}. We want to show that in both cases,
  we
  can separate the parameters from the variable, in the same way we did for
  divisibility conditions. This will be enough to
  conclude. For \ref{case-1}, this is a consequence of Corollary
  \ref{corollary-im}. For \ref{case-2}, we have the following claim.
  \begin{claim}
    Assume we are in \ref{case-2}. Let $\bar{Q}\in\IZ[X]^{n}$ and
    $m\in\IZ\backslash\{0\}$. Then there exists at most one $a_{\bar{Q}}\in A$
    such that $\sum_{i=1}^{n}\op_{Q_{i}}(x_{i})=mx+a_{\bar{Q}}$ has a
    non-degenerate solution in $R(C)\backslash R(A)$.
  \end{claim}
  \begin{proof}
    Assume that there exists another $a'\in A$ that satisfies the claim. Then we
    have $\str{A}\models\Im_{\bar{\op},-\bar{\op}}(a_{\bar{Q}}-a')$. Thus,
    we can find tuples $\bar{b}_{1},\bar{b}_{2}\in R(C)\backslash R(A)$ and
    $\bar{b}'_{1},\bar{b}'_{2}\in R(A)$ such that
    \[\sum_{i=1}^{n}\op_{Q_{i}}(b_{1i})-\op_{Q_{i}}(b_{2i})
      -(\op_{Q_{i}}(b'_{1i})-\op_{Q_{i}}(b'_{2i}))=0.\]
    But this can happen only if $a_{\bar{Q}}=a'$ by Lemma \ref{lemma-sol-eq-orbits}.
  \end{proof}
  As a consequence, we get that in \ref{case-2}, a formula of the form
  $\Im_{\bar{Q},D}(mx+a)$ is in $p_{|\L_{2}}(x)$ if and only if some
  disjunction of formulas of the form
  \[\Im_{\bar{Q}_{I},D}(mx+a_{\bar{Q}_{I}})\wedge
    \Im_{\bar{Q}_{[n]\backslash I},D}(a-a_{\bar{Q}_{I}})\]
  is in $p_{|\L_{2}}(x)$. This proves that the number of types of the form
  $p_{|\L_{2}}(x)$ in \ref{case-2} is at most $\max\{|A|,2^{\aleph_{0}}\}$. We
  conclude that $|S_{1}(A)|\le\max\{|A|,2^{\aleph_{0}}\}$.
\end{proof}
\begin{corr}\label{cor-sup2}
  $\Th(\str{Z})$ is superstable.
\end{corr}
\subsection{Decidability}\label{section-decid}
As a consequence of the fact that the theory of $\str{Z}_{R}$ is axiomatized by
$T_{R}$ when $R(\Z)$ is enumerated by a regular sequence, we get the following
decidability result.
First let us recall that a sequence $(r_{n})$ is \emph{effectively congruence
  periodic} if for all $k\in\IN^{>1}$, there exist effective constants $m,
p\in\IN$ such that the sequence $(r_{n})_{n\ge m}$ is periodic modulo $k$ with
period $p$.
\begin{theorem}\label{theorem-decidable}
  Assume that
  \begin{enumerate}
  \item the limit $\theta=\lim_{n\to\infty}r_{n+1}/r_{n}$ can be computed
    effectively and
  \item $R(\Z)$ is effectively congruence periodic,
  \end{enumerate}
  Then the $\L$-theory $T_{R}$ is decidable.
\end{theorem}
\begin{proof}
  Indeed, under these assumptions, the constants that appear in
  \ref{axiom-operator}, \ref{axiom-eq} can be computed effectively, using the
  proofs of Propositions \ref{proposition-operators} and
  \ref{proposition-op-inhom}. Furthermore, \ref{axiom-image} becomes
  effective thanks to the effective periodicity of $R(\Z)$. Thus, $T_{R}$ is
  recursively axiomatizable. And since $T_{R}$ is complete, we may conclude that
  $T_{R}$ is decidable.
\end{proof}

Examples of regular sequences that satisfy Theorem \ref{theorem-decidable} are
$(q^{n})$, $(n!)$ and the Fibonacci sequence. More generally in \cite{P} a family of such regular sequences is described, namely regular sequences for which in addition $\theta\in \R$, the sequence $(r_{n}/\theta^n)$ has a non zero limit, and $(r_{n}/\theta^n)$ converges to that limit effectively $(\dagger)$ \cite[Proposition 11]{P}.
Furthermore in the case $\theta$ is transcendental, one asks that the sequence is effectively congruence periodic. (When $(r_{n})$ satisfies a linear recurrence relation this is automatic.)
In the appendix of \cite{P}, one can find a proof that if $(r_{n})$ is an A. Bertrand sequence, then condition $(\dagger)$ is indeed fulfilled. (The condition of being an A. Bertrand sequence is a condition on the $\theta$-expansions of real numbers in the interval $[0,1]$ \cite[Definition, section 1]{P}.)

\indent Previously a lot of work has been done on the corresponding expansions of Presburger arithmetic: they are
decidable and admit quantifier elimination (adding to the language in particular the function $\lambda(x)$ which sends a positive $x$ to the biggest element of the sequence smaller than $x$, see below). These results are mainly due to
A. L. Sem\"enov who introduced the notion of
\emph{(effectively) sparse} sequences \cite{Sem}.
Let us introduce some notations: $=_{pp}$ (resp. $>_{pp} , <_{pp} $) means
equality (resp. $>$, $<$) for all but finitely many.\\
\indent A sequence $(r_{n})$ is \emph{sparse} if it satisfies the following properties:
\begin{enumerate}

\item
for all $Q\in\IZ[X]$, either $\{n\in \IN\mid \op_{Q}(n)=0\}=\N$, or
$\op_{Q}>_{pp} 0$, or $\op_{Q}<_{pp} 0$

\item for all $Q\in\IZ[X]$, if $\op_{Q}>_{pp} 0$, then
  there exists a natural number $\Delta$ such that $\op_{Q}(n+\Delta)-r_{n}>0$ for
  all $n\in\IN$.
\end{enumerate}
A sparse sequence $(r_{n})$ is \emph{effectively sparse} if the above conditions
are effective.
\par Let $\str{N}_{<,R}=(\IN,+,-,0,R,<)$.
A. L. Sem\"enov showed that $\str{N}_{<,R}$ is model
complete whenever $R$ is enumerated by a \emph{sparse sequence} and decidable
whenever $R$ is enumerated by an \emph{effectively sparse sequence} \cite[Theorem 3]{Sem}.
 Another proof of \cite[Theorem 3]{Sem} can be found in
\cite{P}. There, one axiomatizes the theory of $\str{N}_{<,R}$ for (almost) sparse sequences and proves it admits quantifier elimination \cite[Proposition
9]{P}, under the assumption that the (almost) sparse sequence is congruence
periodic. While this is an extra assumption in comparison to \cite[Theorem
3]{Sem}, we believe it can be removed at the cost of adding the symbols
$\Im_{[Q],D}$ we used in our quantifier elimination result.
Then decidability follows under the assumption that the sequence is effectively sparse \cite[Proposition 11]{P}.\\
\indent We end this section by giving a relation between sparse and regular
sequences.
\begin{lemma}\label{lem:regular}
 Let $(r_{n})$ be a regular sequence, then $(r_{n})$ is sparse.
\end{lemma}
\begin{proof}
  Let $\theta=\lim\limits_{n\to\infty} r_{n+1}/r_{n}\in \R^{>1}_{\infty}$. It is
  known that $(r_{n})$ is sparse whenever $\theta=\infty$  \cite[§ 3]{Sem} or
  $\lim\limits_{n\to\infty}r_{n}/\theta^{n}\in\IR^{>0}$ \cite[§4]{P}.

  Assume that $\theta\in\IR^{>1}$. We use the same kind of reasoning as in
  Section \ref{section-regseq} (see for instance Lemma \ref{lemma-op-fin}), but we consider
  polynomials $Q(X)\in \IQ[X]$ (instead of in $\Z[X]$) and in addition we
  whether the operators are strictly positive.

  Note that conditions such as
  $\op_{Q}(n)>0$ for all sufficiently large $n\in\IN$ are equivalent to
  $Q(\theta)>0$. So the first condition of the definition of sparse sequence is
  satisfied. For the second condition assume that $Q(\theta)>0$. The condition
  $\op_{Q}(n+\Delta)-r_{n}>0$ for all $n\in\IN$ sufficiently large is equivalent
  to $\theta^{\Delta}Q(\theta)-1>0$. Since $Q(\theta)>0$, we can then find
  $\Delta_{0}\in\IN$ and $n_{0}\in\IN$ such that $\op_{Q}(n+\Delta_{0})-r_{n}>0$
  for all $n\ge n_{0}$. So, letting $\Delta=\Delta_{0}+n_{0}$, we get that a regular
  sequence satisfies the second condition of the definition of a sparse sequence.
\end{proof}
\subsection{Expansions of Presburger arithmetic and the NIP
  property}\label{section-nip}
Let $R$ be enumerated by a regular and congruence periodic sequence. In this last section, we will show that the theory of $\str{Z}_{<,R}=(\IZ,+,-,0,R,<)$ is NIP. This result was announced in \cite[Page 5934]{aschenbrenner} and \cite[Page 359]{aschenbrenner2013} for $(2^{n})$ and the Fibonnaci sequence but without a
proof. So we thought useful to provide one here.\\
\indent We begin by briefly describing the setting of \cite{P} but putting ourselves in $\Z$ (instead of $\N$) in order to remain in the same setting as in the first parts of this article. So from now on, even though all the results in \cite{P} and \cite{Sem} are stated for expansions of $\N$ by a (congruence periodic almost) sparse sequence, we will apply them to the corresponding expansions of $\IZ$.
In \cite{P}, the second author considered theories $T_{<,R}$ extending
Presburger arithmetic and she showed on one hand that $T_{<,R}$ admits quantifier
elimination \cite[Proposition 9]{P} and on the other hand that given an (almost)
sparse congruence periodic sequence $(r_{n})$, the $\L_{<,R}$ expansion
$\str{Z}_{<,R}$ is a model of $T_{<,R}$. Let us briefly recall the
axiomatization of  $T_{<,R}$.

It is convenient to use the same language as in \cite{P}, namely
\[\L_{<,R}=\{+,-,<,0,1,\cdot/n,\lambda_{R},S,S^{-1}\mid  n\in\IN^{>1}\},\]
where $\cdot/n$ is a unary function symbol interpreted as
$\forall x\forall y\; (x/n=y\leftrightarrow\bigvee_{0\le k<n}x=ny+k)$, $n\in\IN^{>1}$, $\lambda_{R}$ is
a unary function symbol interpreted on $\IZ$ as the function sending $x\leq 0$ to $0$ and for $x>0$ to the
biggest element of $R$ smaller than or equal to $x$ and $S$, $S^{-1}$ are unary function symbols
interpreted, as before, as the successor and predecessor functions on $R$. Also in this last section, $\L_g$ will
denote the language $\{+,-,0,1,\cdot/n \mid  n\in\IN^{>1}\}$ (previously we used the unary relation symbols $D_n$ instead of the
unary functions $\cdot/n$.)

The $\L_{<,R}$-theory $T_{<,R}$  \cite[page 1353]{P} consists of a list of axioms translating the properties of $(\Z,+,-,<,0,1)$, of the sequence $(R(\Z),<,1,S, S^{-1})$, and of the relationships between the two structures, namely the properties of the unary function $\lambda_{R}$ (the sequence $R(\Z)$ is interpreted by $\{z\in \Z\mid\,\lambda_{R}(z)=z\}$):
\[
\forall x\;\big{(}(\lambda_{R}(\lambda_{R}(x))=\lambda_{R}(x)\;\wedge\;(x\leq 0\rightarrow \lambda_{R}(x)=0\;\wedge \forall x\geq 1\;
(\lambda_{R}(x)\leq x<S(\lambda_{R}(x))\;\wedge
\]
\[
\forall y (\lambda_{R}(x)\leq y<S(\lambda_{R}(x))\rightarrow \lambda_{R}(y)=\lambda_{R}(x))\big{)},
\]
how the sequence $R(\Z)$ behaves with respect to the congruences and finally the properties of an (almost) sparse sequence (called axiom (6) in \cite{P}). In the following we will make explicit use of this last scheme of axioms and so we will state it below explicitly. We use the notations of \cite{P} and in particular the $\L_{<,R}$-terms that we will describe are the analogs of the $\L_g\cup\L_S$-terms $\op_Q$, for $Q(X)\in \Z[X]$ in Section \ref{sec:axio}.
\par Given an $\L_{<,R}$-term $T(x)$ of the form $\sum_{j\geq 0} m_jS^{-j}(x)/n_j$ with $m=(m_j)$ and $n=(n_j)$ with $m_j \in \IZ, n_j \in \N, m_0\neq 0$, we associate finitely many expressions: $A(m,n)(x)-d(x)$, where $d(x)\in \IZ$, $A(m,n)(x)=\sum_{j\geq 0} \frac{m_j}{n_j}S^{-j}(x)$ and $d(x)=\sum_{j\geq 0} \frac{m_{j}}{n_{j}}d_{j}$
 where $S^{-j}(x)\equiv_{n_{j}} d_{j}$, $0\leq d_{j}<n_{j}$. So even though in $A(m,n)(x)$ coefficients in $\IQ$ may occur, we only apply division by a non-zero natural number $n_{j}$ when the numerator is divisible by $n_{j}$.

Given any such $A(m,n)(x)$ abbreviated by $A(x)$, there is a constant $k(m,n)$ and a natural number $s$ such that the following axiom holds:\\
\indent whenever for some $x\in R$, $x>k(m,n)$ and $A(x)>0$, then  we have either:
 \begin{align}
 \forall y\in R&\; (y>k(m,n)\rightarrow S^{s-1}(y)<A(y)<S^s(y)), \;\;{\rm or}\label{4}\\
\forall y\in R&\; (y>k(m,n)\rightarrow S^{s-1}(y)=A(y)).\label{5}
\end{align}
Moreover, we can describe the behaviour of $A(x)$ under a {\it small} perturbation: there is a constant $k'(m,n)\geq k(m,n)$ and $j_{0}\in \IN$ such that
for every $y\in R, y>k'(m,n)$ and all $j\geq j_{0}$:
 \begin{align}
S^{s-1}(y)\leq A(y)<S^s(y)&\rightarrow S^{s-1}(y)\leq A(y)+S^{-j}(y)<S^s(y), \label{6}\\
S^{s-1}(y)< A(y)<S^s(y)&\rightarrow S^{s-1}(y)\leq A(y)-S^{-j}(y)<S^s(y), \label{7}\\
S^{s-1}(y)=A(y)&\rightarrow S^{s-2}(y)\leq A(y)-S^{-j}(y)<S^{s-1}(y). \label{8}
\end{align}
We will refer to this scheme of axioms (when $A(m,n)(x)$ varies) by $(S_{sparse})$.
\begin{remark}
  The axiom scheme $(S_{{}sparse})$ has the following consequence. Let $M\models
  T_{<,R}$. Recall that two strictly positive elements $g, h\in M$
are in the same archimedean class if there are $n, m\in \N^0$ such that $g\leq
nh\leq mg$. We observe that if the sequence $(r_{n+1}/r_{n})$ is unbounded, then
$r_{n+1}/r_{n}\to\infty$. Indeed, if $(r_{n+1}/r_{n})$ is unbounded, then for
all $m\in\IN^{>0}$, for all but finitely many elements $x$ of $R$, by axioms
(\ref{4}) and (\ref{5}), we have  $mx<S(x)$. So we get that either there exists
$n\in \N^{>0}$
such that for all positive $y$ but finitely many $\lambda_R(y)\leq y\leq
n\lambda_R(y)$ or for all $m\in \N^{>0}$ for all but finitely many $y$,
$m\lambda_R(y)\leq S(\lambda_R(y))$. In other words for elements $y$ bigger than
$\IZ$, either $y$ and $\lambda_R(y)$ are in the same archimedean class or never.
\end{remark}
\begin{lemma}\label{lem:sparse}
 Let $R(\Z)$ be enumerated by a sparse sequence $(r_{n})$, then $\str{Z}_{<,R}$
 satisfies the scheme $(S_{sparse})$.
 \end{lemma}
\begin{proof}
Let $\op_{Q}(n)=a_{0}r_{n}+\cdots+a_{d}r_{n+d}$ with $Q(X)=\sum_{i=0}^d a_i X^i\in \IQ[X]$.
Note that, for $m\in \IN$, $\op_{QX^{m}}(n)= a_{0}r_{n+m}+\cdots+a_{d}r_{n+d+m}$, $\op_{QX^{m}-1}(n)=a_{0}r_{n+m}+\cdots+a_{d}r_{n+d+m}-r_{n}.$

Assume that $\op_{Q}>_{pp}0$. Since $(r_n)$ is sparse, there exists $\Delta\in \IN$ such that
$\op_{Q}(n+\Delta)-r_{n}>0$ for all $n$.  Then choose $\Delta$ minimal with the property that $\op_{QX^{\Delta}-1}>_{pp} 0$. This implies in particular that we do not have
$\op_{QX^{\Delta-1}-1}>_{pp} 0$. But $\op_{QX^{\Delta-1}-1}$ is another operator
and so since the sequence is sparse, we have that either $\op_{QX^{\Delta-1}-1}(n)=0$ for all $n$, or that $\op_{QX^{\Delta-1}-1}<_{pp}0$.
 So for almost all $n$, we have that $r_{0}\leq a_{0}r_{n+\Delta}+\cdots+a_{d}r_{n+d+\Delta}< r_{n+1}.$
 In case $a_{0}r_{n+\Delta}+\cdots+a_{d}r_{n+d+\Delta}-r_{n}>0$, by re-applying the same argument we get that for some $\Delta'$,
 we have that $a_{0}r_{n+\Delta+\Delta'}+\cdots+a_{d}r_{n+d+\Delta+\Delta'}-r_{n+\Delta'}-r_{n}>0$, namely setting $m=n+\Delta'$, we get:
 $a_{0}r_{m+\Delta}+\cdots+a_{d}r_{m+d+\Delta}-r_{m-\Delta'}>r_{m}.$
 Finally, we consider $r_{n+1}-a_{0}r_{n+\Delta}+\cdots+a_{d}r_{n+d+\Delta}$, since this is strictly positive there exists $\Delta''$ such that
 $r_{n+\Delta''+1}-(a_{0}r_{n+\Delta+\Delta''}+\cdots+a_{d}r_{n+d+\Delta+\Delta''})-r_{n}>0$. Therefore, $a_{0}r_{n+\Delta+\Delta''}+\cdots+a_{d}r_{n+d+\Delta+\Delta''}+r_{n}<r_{n+\Delta''+1}$. Again setting $m=n+\Delta''$, we get:
 $a_{0}r_{m+\Delta}+\cdots+a_{d}r_{m+d+\Delta}+r_{m-\Delta''}<r_{m+1}.$
 \end{proof}

\par Let $T$ be a complete $\L$-theory. Then $T$ is NIP if all (partitioned)
formulas $\varphi(x; y)$ are NIP \cite[Definition 2.10]{simon}. (In a
partitioned formula, one indicates the parameters (in this case $y$)).  Also,
for convenience, in this section we will adopt the following conventions: we
will
use single letters to possibly denote tuples of variables and since we deal with
ordered structures, we will use the notation $|x|$ for $\max\{x,-x\}$, even
though we previously used it for denoting either the cardinality of a set or the length of a tuple.
\begin{lemma}[{\cite[Lemma 2.9]{simon}}] \label{lem:bc} Assume that $T$ admits quantifier elimination. Then $T$ is NIP if and only if  all atomic formulas $\varphi(x; y)$ are NIP.
\end{lemma}
\begin{lemma}[{\cite[Proposition 2.8]{simon}}] Let $\str{M}\models T$.
 The partitioned formula $\varphi(x; y)$ is NIP if and only if for any
 indiscernible sequence $(a_{i}\mid i<\omega_{1})$, where the length of $a_i$ is equal to the length of $x$ and tuple $b$ in $M$ there is some end segment $I\subset \omega_{1}$ such that  for  any $i\in I$, the truth value of $\varphi(a_{i};b)$ is constant.
\end{lemma}
We will use Hahn representation theorem for divisible ordered abelian groups \cite[Section 4.5]{glass}. Let $G$ be an abelian totally ordered group and let $\bar G$ be its divisible closure. Given an element $g\in G\setminus\{0\}$ there is a unique convex subgroup $V$ maximal for the property of not containing $g$; it is called a {\it value} for $g$. There is also a smallest convex subgroup $V^+$ containing $g$ and the quotient $V^+/V$ is an archimedean ordered group (which by H\"older's theorem, embeds in $(\IR,+,<,0)$). The set of all values in $G$ forms a chain denoted by $\Gamma(G)$; we set $\Gamma(G)=\{V_{\gamma}\mid \gamma\in \Gamma\}$ and for $V=V_\gamma$, we denote $V^+$ by $V^{\gamma}$.
Note that $\Gamma(\bar G)=\Gamma(G)$. Denote by $R_{\gamma}=V^{\gamma}/V_{\gamma}$
and let $\bar R_{\gamma}$ be the $\IQ$-vector-subspace generated by $R_\gamma$ in $\R$.
 One can decompose $\bar G$ as a direct sum $\bar G=\bar V^{\gamma}\oplus D_{\gamma}$, where $\bar V^{\gamma}$ is the divisible closure of $V^{\gamma}$ in $\bar G$ and $D_\gamma$ is some direct summand. Denote by $\pi_{\gamma}$ the projection of $\bar G$ to $\bar V^{\gamma}$ and let $\rho_{\gamma}:\bar V^{\gamma}\to \bar R_{\gamma}$. Then one sends $g$ to the function $\hat g: \Gamma(G)\to \IR:\gamma\mapsto \rho_{\gamma}\pi_{\gamma}(g)=\hat g(\gamma)$.
One verifies that $\supp(g)=\{\gamma\in \Gamma(G)\mid \hat g(\gamma)\neq 0\}$ is an anti-well ordered subset of $\Gamma(G)$.
Denote by $V(\Gamma(G),\bar R_{\gamma})$ the lexicographically ordered group of functions $f$ from $\Gamma(G)$ to $\R$ with anti-well-ordered support, such that $f(\gamma)\in \bar R_\gamma$, for any $\gamma\in \Gamma(G)$. Then $G$ embeds in $V(\Gamma(G),\bar R_{\gamma})$ by the map $g\mapsto \hat g$ \cite[Theorem 4C]{glass}.
Define a map $v:G\to \Gamma(G)$ which sends $g\in G\setminus\{0\}$ to
$\max(\supp(\hat{g}))$. This is a valuation map on $G$ as defined in \cite[Chapter 4, section 4]{Fuchs}  except that there one takes
the opposite order on $\Gamma(G)$. It is constant on an archimedean class, namely if satisfies the following:
for all $g, h\in G^{>0}$, $v(g)=v(h)$ if and only if $g\le  nh \le mg$ for some
$n, m\in\IN^{>0}$ (in other words, $g$ and $h$ are in the same archimedean class).

\indent In order to show that $T_{<,R}$ is NIP, we need some preparatory work in
order to evaluate terms of the form $\lambda_R(x\pm y)$ for $x, y>0$.
\begin{lemma}\label{lemma-lambda}
  Let $\str{M}\models T_{<,R}$, $d\in M$ and $(c_{i}\mid i\in \omega_1)$ be a
  non-constant indiscernible sequence in $M$ such that $d> 0$ and $c_{i}>0$ for all
  $i\in \omega_1$. Then there exist $i_{0}\in \omega_1$, $\ell\in\IZ$ such that one of the following
  holds for all $i\ge i_{0}$:
  \begin{itemize}
  \item $\lambda_R(c_{i}\pm d)=S^{\ell}(\lambda_R(c_{i}))$;
  \item $\lambda_R(d\pm c_{i})=S^{\ell}(\lambda_R(d))$;
  \item $\lambda_R(|d-c_{i}|)= S^{\ell}(\lambda_R(|d-c_{i_{0}}|))$;
  \item $\lambda_R(|d-c_{i}|)= S^{\ell}(\lambda_R(|c_{i+1}-c_{i}|))$;
  \item $\lambda_R(|d-c_{i+1}|)=S^{\ell}(\lambda_R(|c_{i+1}-c_{i}|))$.
  \end{itemize}
\end{lemma}
\begin{proof}
  The proof has two main ingredients. The first one is that
  in a  model of a NIP theory, an indiscernible sequence remains eventually
  indiscernible over a parameter (see \cite[Claim in the proof of
 Proposition 2.11]{simon}). The second one is the following consequence of
 axioms (\ref{4}) and (\ref{5}) in scheme $(S_{sparse})$: for all $n,m\in\IN^{>0}$
 there exist $\ell,\ell'\in\IZ$ such that  for
 all $x\in M$, if $x$ is bigger than some natural number then
 $S^{\ell}(\lambda_{R}(x))\le\frac{1}{n}\lambda_{R}(x)\le
 \frac{m}{n}S(\lambda_{R}(x))\le S^{\ell'}(\lambda_{R}(x))$.

 Let $(c_{i}\mid i\in \omega_1)$ be a
  non-constant indiscernible sequence. We will apply the first ingredient to
  certain sequences of the form $(t(c_{i})\mid i\in \omega_1)$, where $t(x)$ is
  a  $\L_{<,R}$-terms, and to the NIP
  theory $\Th(\IZ,+,-,0,<)$.
  For ease of notations, we may assume that
  $(c_{i}\mid i\in \omega_1)$ is indiscernible over $d$ in $\{+,-,0,<\}$.  In particular we have that either
  $d>c_{i}$ for all $i\in \omega_1$ or $d<c_{i}$ for all $i\in \omega_1$.
  By comparing $v(d)$ to $v(c_0)$, we obtain that the sequence $(c_i\mid i\in \omega_1)$ falls into the following cases for some $m, n\in \N^{>0}$:
  \begin{enumerate}
  \item $c_{i}\le n(c_{i}\pm d)\le m c_{i}$ for all $i\in \omega_1$;
  \item $d\le n(d\pm c_{i})\le m d$ for all $i\in \omega_1$;
  \item $|c_{i+1}-c_{i}|\le n|d-c_{i+1}|\le m|c_{i+1}-c_{i}|$ for all $i\in \omega_1$;
  \item $|d-c_{0}|\le n|d-c_{i}|\le m|d-c_{0}|$ for all $i\in \omega_1$;
  \item $|c_{i+1}-c_{i}|\le n|d-c_{i}|\le m|c_{i+1}-c_{i}|$ for all $i\in \omega_1$.
  \end{enumerate}
 Cases 1 and 2 occur when $v(d)\neq v(c_0)$ and case 2 together with the remaining cases when $v(d)=v(c_0)$. In case $v(d)=v(c_0)$, we further compare
 $v(c_{1}-c_{0})$, $v(d-c_{1})$ and $v(d-c_0)$. Let us check it in details.
  \begin{enumerate}
 \item $v(d)<v(c_{0})$. Then $v(c_{0}\pm d)=v(c_{0})$, so that $c_{0}\le
   n(c_{0}\pm d)\le m c_{0}$ for some $n,m\in\IN^{>0}$. By indiscernibility over
   $d$, we get that $c_{i}\le n(c_{i}\pm d)\le m c_{i}$ for all $i\in \omega_1$;
 \item $v(d)>v(c_{0})$. Then $v(d\pm c_{0})=v(d)$, so that $d\le
   n(d\pm c_{0})\le m d$ for some $n,m\in\IN^{>0}$. By indiscernibility over
   $d$, we get that $d\le n(d\pm c_{i})\le m d$ for all $i\in \omega_1$;
 \item $v(d)=v(c_{0})$. Then, as $v(d+c_{0})=v(d)$, we have that $d\le
   n(d+c_{0})\le md$ for some $n,m\in\IN^{>0}$. Hence  $d\le n(d + c_{i})\le m
   d$ for all $i\in \omega_1$. Furthermore, one of the following holds, using
   indiscernibility as before:
   \begin{enumerate}
   \item $v(c_{1}-c_{0})=v(d-c_{1})$. Then there are $n, m\in\IN^{>0}$ such that
   $|c_{1}-c_{0}|\le n|d-c_{1}|\le m|c_{1}-c_{0}|$ and so for all $i\in \omega_1$,
     $|c_{i+1}-c_{i}|\le n|d-c_{i+1}|\le m|c_{i+1}-c_{i}|$;
   \item $v(c_{1}-c_{0})\neq v(d-c_{1})$.  Then
     $v(d-c_{0})=\max\{v(d-c_{1}),v(c_{1}-c_{0})\}$. So,
     \begin{enumerate}
     \item either $v(d-c_{0})=v(d-c_{1})$, in which case
       there are $n, m\in\IN^{>0}$ such that $|d-c_{0}|\le n|d-c_{1}|\le
       m|d-c_{0}|$ and so for all $i\in \omega_1$,
       $|d-c_{0}|\le n|d-c_{i}|\le
       m|d-c_{0}|$;
     \item or $v(d-c_{0})=v(c_{1}-c_{0})$, in which case
       there are $n, m\in\IN^{>0}$ such that $|c_{1}-c_{0}|\le n|d-c_{0}|\le m|c_{1}-c_{0}|$ and so for all $i\in \omega_1$,
       $|c_{i+1}-c_{i}|\le n|d-c_{i}|\le m|c_{i+1}-c_{i}|$.
     \end{enumerate}
   \end{enumerate}
 \end{enumerate}
 Now given $g, h\in M^{>0}$ in the same archimedean class, more precisely such that
 $h\le ng\le mh$, for some $n, m\in\IN^{>0}$, observe that there are $\ell, \ell'\in \N$
 such that $\lambda_R(g)\in\{S^{k}(\lambda_R(h))\mid \ell\le k\le\ell'\}$. This
 follows from axioms (\ref{4}) and (\ref{5}). Indeed,
 there are $\ell, \ell'\in \N$ such that  to $\frac{m}{n}S(\lambda_R(h))\leq S^{\ell}(\lambda_R(h))$ and $S^{\ell'}(\lambda_R(h))\leq \frac{1}{n}\lambda_R(h)$,
 for $h$ bigger than some natural number.

  Applying the discussion above, with $g$ of the form $\vert d\pm c_i\vert$  and $h$ equal to either $c_i$, $d$, $d-c_0$, $\vert c_{i+1}-c_i\vert$, or $\vert c_{i}-c_{i-1}\vert$, we get:
  \begin{itemize}
  \item $\lambda_R(c_{i}\pm d)\in\{S^{k}(\lambda_R(c_{i}))\mid \ell\le k\le\ell'\}$;
  \item $\lambda_R(d\pm c_{i})\in\{S^{k}(\lambda_R(d))\mid \ell\le k\le\ell'\}$;
  \item $\lambda_R(|d-c_{i}|)\in\{S^{k}(\lambda_R(|d-c_{0}|))\mid \ell\le
    k\le\ell'\}$;
  \item $\lambda_R(|d-c_{i}|)\in\{S^{k}(\lambda_R(|c_{i+1}-c_{i}|))\mid \ell\le
    k\le\ell'\}$;
  \item $\lambda_R(|d-c_{i+1}|)\in\{S^{k}(\lambda_R(|c_{i+1}-c_{i}|))\mid \ell\le
    k\le\ell'\}$.
  \end{itemize}
In case $\lambda_R(|d-c_{i}|)$ and $\lambda_R(d+c_{i})$ belong to a finite set, for $i$ sufficiently big, we get a constant value since these sequences are monotone.

\noindent  Let us focus on the case where $\lambda_R(|d-c_{i}|)\in\{S^{k}(\lambda_R(|c_{i+1}-c_{i}|))\mid \ell\le k\le\ell'\}$. The cases where $\lambda_R(c_{i}\pm d)\in\{S^{k}(\lambda_R(c_{i}))\mid \ell\le k\le\ell'\}$ and $\lambda_R(|d-c_{i+1}|)\in\{S^{k}(\lambda_R(|c_{i+1}-c_{i}|))\mid \ell\le   k\le\ell'\}$ are similar.

First let us assume that
  $|d-c_{i}|=d-c_{i}$ for all $i\in \omega_1$.
  Recall that
  $\lambda_R(d-c_{i})=S^{k}(\lambda_R(|c_{i+1}-c_{i}|))$ means
  $S^{k}(\lambda_R(|c_{i+1}-c_{i}|))\le
  d-c_{i}<S^{k+1}(\lambda_R(|c_{i+1}-c_{i}|))$. So, for $k\in[\ell,\ell']$, let
   $(t_k(c_{i}))\mid {i\in \omega_1})$ be the sequence defined by
  $t_k(c_{i})=S^{k}(\lambda_R(|c_{i+1}-c_{i}|))+c_{i}$. For all $k\in[\ell,\ell']$, the
  sequence $(t_k(c_{i})\mid {i\in \omega_1})$ is indiscernible. Hence for each
  $k\in[\ell,\ell']$, we may assume that
  $(t_k(c_{i})\mid i\in \omega_1)$ is $\{<\}$-indiscernible over $d$, for $i>i_k$.
  Let  $k_{0}\in[\ell,\ell']$ maximal such that $t_{k_0}(c_{i})\le d$
  for all $i>i_0=\max\{i_{k}\mid {\ell\leq k\leq \ell'}\}$.
  Then we have that for all $i>i_0$ that
  $d<t_{k_{0}+1}(c_i)$. Thus,
  $\lambda_R(d-c_{i})=S^{k_{0}}(\lambda_R(|c_{i+1}-c_{i}|))$ for all $i>i_0$.

  Second, if $|d-c_{i}|=c_{i}-d$ for all $i\in \omega_1$, we can repeat the argument using the
  sequences $(t'_{k}(c_i)\mid {i\in \omega_1})$ defined by
  $t'_k(c_{i})=c_{i}-S^{k}(\lambda_R(|c_{i+1}-c_{i}|))$, for all $i\in \omega_1$ and $k\in[\ell,\ell']$.
\end{proof}

\begin{theorem}\label{thm-nip} The theory $T_{<,R}$ is NIP.
\end{theorem}
\begin{proof}
  Since $T_{<,R}$ has quantifier elimination by \cite[Proposition 9]{P}, we may
  apply Lemma \ref{lem:bc} and consider atomic formulas. Let us first look at $\L_{<,R}$-terms, when we evaluate
  them along an indiscernible sequence and a parameter. Let $(a_{i}\mid i\in \omega_1)$ be a
  non-constant
  indiscernible sequence and $b\in M$ some tuple of parameters.
  \begin{claim}\label{claim:sep} Let $t(x;y)$ be an $\L_{<,R}$-term. Then there
    exists  $i_{0},i_{1},\ldots,i_{n}\in \omega_1$ and
    $\L_{<,R}$-terms $t_{1}(x,x_1,\ldots,x_{2n})$,
    $t_{2}(y,y_{1},\ldots,y_{n})$ such that
    for all
    $i\ge i_{0}$:
    \[
      t(a_{i}; b)=
      t_{1}(a_{i},a_{i+1},\ldots,a_{i+n},a_{i-1},\ldots,a_{i-n})
      +t_{2}(b,a_{i_{1}},\ldots,a_{i_{n}}).
    \]
  \end{claim}
\begin{proof}[Proof of Claim]
  We proceed by induction on the number of occurrences of the symbol $\lambda_R$
  in $t(x;y)$. If $t(x;y)$ is an $\L_{g}$-term, we can write it as
  $t_{1}(x)+t_{2}(y)$, see \cite[Lemma 4]{P}. So what remains to be shown is
  that the claim holds for terms of the form
  \[\lambda_R( t_{1}(a_{i},a_{i+1},\ldots,a_{i+n},a_{i-1},\ldots,a_{i-n})
    +t_{2}(b,a_{i_{1}},\ldots,a_{i_{n}})),\]
  where $t_{1}(x,x_1,\ldots,x_{2n})$ and
    $t_{2}(y,y_{1},\ldots,y_{n})$ are $\L_{<,R}$-terms. Let
    $d=t_{2}(b,a_{i_{1}},\ldots,a_{i_{n}})$ and $c_{i}=
    t_{1}(a_{i},a_{i+1},\ldots,a_{i+n},a_{i-1},\ldots,a_{i-n})$. Without loss of
    generality, we may assume that $d\neq 0$ and $(c_{i}\mid i\in\omega_{1})$ is
    non-constant. Then
    $(c_{i}\mid i\in\omega_{1})$
    is indiscernible and one of the following holds:
    \begin{enumerate}
    \item $c_{i}+d\le 0$ for all $i\in \omega_1$ sufficiently large. In that case,
      $\lambda_R(c_{i}+d)=0$ for all $i\in \omega_1$ sufficiently large;
    \item $c_{i}+d> 0$ for all $i\in \omega_1$ sufficiently large. In this case, we
      apply Lemma \ref{lemma-lambda} to conclude.\qedhere
    \end{enumerate}
\end{proof}

Finally let us check that the truth value of $\varphi(a_{i};b)$ is eventually
constant, when $\phi(x;y)$ is an atomic formula. By Claim \ref{claim:sep}, we
may assume that $\varphi(x; y)$ is of the form $t_{1}(\bar x)< t_{2}(\bar{y})$
or $t_{1}(\bar x)=t_{2}(\bar{y})$ for $t_{1}$ and $t_{2}$ two
$\L_{<,R}$-terms. Here we used $\bar x$ instead of $x$ since the length of
$\bar x$ is possibly bigger than the length of $x$ (and likewise for $y$).
Let $c_{i}=t_{1}(\bar{a})$ and $d=t_{2}(\bar{b})$. Then $(c_{i}\mid i\in \omega_1)$ is
indiscernible and there exists $i_{0}\in \omega_1$ such that $(c_{i}\mid i\ge i_{0})$
is indiscernible over $d$ in the language $\{<\}$.
In particular the truth value of $t_{1}(\bar
x)<t_{2}(\bar{b})$ or $t_{1}(\bar x)=t_{2}(\bar{b})$ is constant over
$(c_{i}\mid i\ge i_{0})$.
\end{proof}

\begin{corr}\label{corollary-nip}
  Let $R(\IZ)$ be enumerated by a congruence periodic regular sequence.
  Then the theory $\Th(\str{Z}_{<,R})$ is NIP.
\end{corr}
\begin{proof}
  Since $R(\Z)$ is enumerated by a sparse sequence (see Lemma
  \ref{lem:regular}), by Lemma \ref{lem:sparse}, the theory $T_{<,R}$ is equal
  to the theory of $\str{Z}_{<,R}$.
\end{proof}
\section*{Acknowledgments}
The authors would like to thank Gabriel Conant for
helpful comments on a previous version of this paper and in particular for
spotting a mistake in a previous draft and for enlightening
exchanges. Finally, we would like to warmly thank the referee for their very detailed and careful report.
\bibliographystyle{abbrv}

\end{document}